\documentclass[11pt,a4paper,reqno]{amsart}
\usepackage{amsthm}
\usepackage[table]{xcolor}
\usepackage[utf8]{inputenc}
\usepackage[T1]{fontenc}
\usepackage{lipsum}
\usepackage{mathtools}
\usepackage{tabularx}
\usepackage{booktabs}
\usepackage[left=2.0cm, right=2.0cm, top=2.00cm, bottom=2.00cm]{geometry}
\usepackage{amsmath}
\usepackage{helvet}
\usepackage{amsfonts}
\usepackage{amssymb}
\usepackage{graphicx}
\usepackage{mathrsfs} 
\allowdisplaybreaks
\def\d{{\rm d}v_{g}}

\def\weak{\rightharpoonup}
\def\embd{\hookrightarrow}
\usepackage{cancel}
\usepackage{xcolor}
\definecolor{bluecite}{HTML}{0875b7}
\usepackage[unicode=true,
bookmarksopen={true},
pdffitwindow=true, 
colorlinks=true, 
linkcolor=bluecite, 
citecolor=bluecite, 
urlcolor=bluecite, 
hyperfootnotes=false, 
pdfstartview={FitH},
pdfpagemode= UseNone]{hyperref}
\numberwithin{equation}{section}
\numberwithin{figure}{section}

\newtheorem{theorem}{Theorem}[section]
\newtheorem{lemma}{Lemma}[section]

\newtheorem{example}{Example}[section]
\newtheorem{remark}{Remark}[section]

\newcommand{\ds}{\displaystyle}
\numberwithin{equation}{section}

\usepackage{xcolor}
\definecolor{acsblue}{RGB}{17,76,139}

\title[Sobolev compactness versus orbit expansions]{Compact Sobolev embeddings on non-compact manifolds via orbit expansions of isometry groups
}
\author{Csaba Farkas}
\address{Sapientia Hungarian University of Transylvania\\ Department of Mathematics and Computer Science, Tg. Mure\c s, Romania.}
\email{farkascs@ms.sapientia.ro, farkas.csaba2008@gmail.com}
\author{Alexandru Krist\'aly}
\address{Babe\c s-Bolyai University, Department of Economics, Cluj-Napoca, Romania.\\
Institute of Applied Mathematics, \'Obuda University, Budapest, Hungary.}
\email{alex.kristaly@econ.ubbcluj.ro, kristaly.alexandru@nik.uni-obuda.hu}

\author{\'Agnes Mester}
\address{Institute of Applied Mathematics, \'Obuda University, Budapest, Hungary.}
\email{mester.agnes@stud.uni-obuda.hu}
\usepackage{abstract}
\usepackage{mdframed}
\makeatletter
\renewenvironment{abstract}{%
	\ifx\maketitle\relax
	\ClassWarning{\@classname}{Abstract should precede
		\protect\maketitle\space in AMS documentclasses; reported}%
	\fi
	\global\setbox\abstractbox=\vtop \bgroup
	\normalfont
	\list{}{\labelwidth\z@
		\leftmargin0pc \rightmargin\leftmargin
		\listparindent\normalparindent \itemindent\z@
		\parsep\z@ \@plus\p@
		
	}%
	\item[]\vskip-\baselineskip
	\begin{mdframed}[backgroundcolor=gray!20,hidealllines=true]
		\item[\hskip\labelsep\scshape{\color{acsblue}\normalfont\textbf{Abstract:}}]%
	}{%
	\end{mdframed}\endlist\egroup
	\ifx\@setabstract\relax \@setabstracta \fi
}

\begin{document}
\maketitle
\begin{abstract}
	Given a complete non-compact  Riemannian manifold $(M,g)$ with certain curvature restrictions, we introduce an expansion condition concerning a group of isometries $G$ of $(M,g)$ that characterizes the coerciveness of $G$ in the sense of Skrzypczak and Tintarev (Arch. Math., 2013). Furthermore, under these conditions, 	 compact Sobolev-type embeddings \`a la Berestycki-Lions are proved for the full range of admissible parameters (Sobolev, Moser-Trudinger and Morrey). 
We also consider the case of non-compact Randers-type Finsler manifolds with finite reversibility constant inheriting similar embedding properties as their Riemannian companions; sharpness of such constructions are shown by means of the Funk model. 
As an application, a quasilinear PDE  on Randers spaces is studied by using the above compact embeddings and variational arguments. 

	\vspace*{0.3cm}
	
	\noindent {\color{acsblue}\normalfont\textbf{Keywords:}} Sobolev embeddings; compactness;  Riemannian/Finsler manifolds; isometries. \\
	
		\noindent {\color{acsblue}\normalfont\textbf{2020 Mathematics Subject Classification:}} 58J05, 53C60, 58J60.
\end{abstract}

	\section{Introduction and main results}\label{intro}
	Compact Sobolev embeddings turn out to be fundamental tools in the study
	of variational problems, being frequently used to study the existence of solutions to elliptic equations, see e.g. Willem \cite{Willem}. More precisely, they are used for proving essential properties of the energy functionals associated with the studied problems (such as sequential lower semicontinuity or the Palais–Smale condition), in order to apply certain minimization and/or minimax arguments.
	
	If $\Omega\subseteq \mathbb R^d$ is an open set with sufficiently smooth boundary in the Euclidean space $\mathbb{R}^d$, it is well known that the Sobolev space $W^{1,p}(\Omega)$ can be continuously embedded into the Lebesgue space $L^q(\Omega),$ assuming  the parameters $p$ and $q$ verify the range properties: (i) $p\leq q\leq p^*:=\frac{pd}{d-p}$ if $p<d;$ (ii) $q\in [p,+\infty)$ if $p=d$, and (iii)  $q=+\infty$ if $p>d$. 
	%
	On one hand, when $\Omega$ is \textit{bounded}, due to the Rellich-Kondrachov  theorem, the previous embeddings are all compact injections, see Brezis \cite{Brezis}. On the other hand, when $\Omega$ is \textit{unbounded},  the aforementioned compactness need not 
	hold, see Adams and Fournier \cite{Adams}; for instance, if $\Omega=\mathbb R^d$, the dilation and translation of functions preclude such compactness phenomena. However, symmetries  may recover compactness; indeed, it was proved by Berestycki--Lions (see Berestycki and Lions \cite{BL}, Lions \cite{Lions}, and also Cho and Ozawa \cite{KinaiakSzimm}, Ebihara and Schonbek \cite{JapanBeagyazas}, Strauss \cite{Strauss}, and Willem \cite{Willem}) that if $\ p\leq d$ then the embedding $W^{1,p}_{\mathrm{rad}}(\mathbb{R}^d)\hookrightarrow L^q(\mathbb{R}^d)$ is compact  whenever $p<q<p^*$, where $W^{1,p}_{\mathrm{rad}}(\mathbb{R}^d)$ stands for the subspace of radially symmetric functions of $W^{1,p}(\mathbb{R}^d)$, i.e. $$W^{1,p}_{\mathrm{rad}}(\mathbb{R}^d)=\left\{u\in W^{1,p}(\mathbb{R}^d): u(\xi x)=u(x)\ \mbox{ for all }\xi\in O(d)\right\},$$  
	where $O(d)$ is the  orthogonal group in $\mathbb R^d$. 
	In the case of Morrey--Sobolev embeddings, it turns out that  $W^{1,p}_{\rm rad}(\mathbb{R}^d)$ can be also compactly embedded into $L^\infty(\mathbb{R}^d)$ when $2\leq d<p<+\infty$, see Krist\'aly \cite{KristalyJDE2006}. 
	
	Geometrically, Berestycki–Lions' compactness  is based on a careful estimate of the  functions at infinity. One first observes that the maximal number of mutually disjoint balls having a fixed radius and centered on the orbit  $\{\xi x:\xi \in O(d)\}$ tends to infinity whenever $|x|\to \infty;$ this phenomenon is similar to the maximal number of disjoint patches with fixed diameter on a balloon with continuous expansion. Now, the latter expansiveness property of the balls combined with the invariance of the Lebesgue measure w.r.t. translations implies that the radially symmetric functions rapidly decay to zero at infinity; this fact is crucial to recovering compactness of Sobolev embeddings on  unbounded  domains, see e.g. Ebihara and  Schonbek \cite{JapanBeagyazas}, Krist\'aly \cite{KristalyJDE2006} and Willem \cite{Willem}; moreover, this argument is in full concordance with the initial approach of Strauss \cite{Strauss}.

	Notice that a  Berestycki–Lions-type theorem has  been  established on Riemannian manifolds by Hebey and Vaugon \cite{HebeyVaugon}, see also Hebey \cite[Theorems 9.5 \& 9.6]{Hebey-konyv1}. More precisely, if $G$ is a compact subgroup of the group of global isometries of the complete Riemannian manifold $(M, g)$, then (under additional assumptions
	on the geometry of $(M, g)$ and  on the orbits under the action of $G$)                                                                                                                                                                                                                                                                                                                                                            
	the embedding $W^{1,p}_G(M)\hookrightarrow L^q(M)$ is compact, where $W^{1,p}_G(M)$ denotes the set of $G$-invariant functions of $W_g^{1,p}(M)$.  Berestycki–Lions-type compactness results have been extended to non-compact metric measure spaces as well, see G\'orka \cite{Gorka}, and generalized to Lebesgue–Sobolev spaces $W_G^{1,p(\cdot)}(M)$ in the setting of complete Riemannian manifolds, see Gaczkowski, G\'{o}rka and Pons \cite{LengyelekJFA} and Skrzypczak \cite{Skrzypczak}.  	
	                                                                                                                                                                                                               
 Skrzypczak and Tintarev \cite{SkrzypczakTintarev, TintarevKonyv} identified general geometric conditions that are behind the compactness of Sobolev embeddings of the type $W^{1,p}_G(M)\hookrightarrow L^q(M)$ for certain ranges of $p$ and $q$; their studies deeply depend  on the curvature of the Riemannian manifold. In the light of their works,                                          
	our purpose is twofold; namely, we provide an alternative characterization of the properties described by Skrzypczak and Tintarev \cite{SkrzypczakTintarev, TintarevKonyv} by using the expansion of geodesic balls and state the compact Sobolev embeddings of isometry-invariant Sobolev functions to Lebesgue spaces for the full admissible range of parameters. Given $d\in \mathbb N$ with $d\geq 2$, we say that $(p,q)\in (1,\infty)\times (1,\infty]$ is a $d$-\textit{admissible pair} whenever 
	\begin{itemize}
		\item[(\textbf{S}):]\label{S} $p<q<p^*=\frac{pd}{d-p}$ if $1<p<d$ (Sobolev-type);
		\item[(\textbf{MT}):]\label{MT} $q\in (p,\infty)$ if $p=d$ (Moser-Trudinger-type);
		\item[(\textbf{M}):]\label{M} $q=+\infty$ if $p>d$ (Morrey-type).
	\end{itemize}

 	 In order to present our results, let $(M,g)$ be a complete $d$-dimensional  Riemannian manifold, and let 
	 $d_g:M\times M\to [0,\infty)$ be the distance function induced by the Riemannian metric $g$.  
	 Denote by $\mathrm{Isom}_g(M)$ the isometry group of the manifold $(M,g)$. It is well-known that $\mathrm{Isom}_g(M)$ is a Lie group with respect to the compact open topology and it acts differentiably on $M$. Let $G$ be a compact connected subgroup of $\mathrm{Isom}_g(M)$. In the sequel, we denote the action of an element $\xi \in G$ by $\xi x \coloneqq \xi(x)$ for every  $x \in M$.   Let $$\mathrm{Fix}_M(G)=\{x\in M:\xi x = x\ \mbox{for all } \xi\in G\}$$ be the \textit{fixed point set} of $G$ on $M$. 
	 Denote by $\mathcal{O}_G^x = \{\xi x: \xi \in G\}$ the $G$-\textit{orbit} of the point $x\in M$. 
	 The subspace of $W_g^{1,p}(M)$ consisting by $G$-\textit{invariant function}s is 
	 \begin{equation*}
	 W^{1,p}_G(M)=\left\{u\in W^{1,p}_g(M): u\circ \xi =u\ \mbox{ for all }\xi \in G\right\}.
	 \end{equation*}
	 Since $G$ is a subgroup of isometries,  $W^{1,p}_G (M)$ turns out to be a closed subspace of $W_g^{1,p}(M)$. We say that a continuous action of a group $G$ on a complete Riemannian manifold $M$ is \textit{coercive} (see Tintarev \cite[Definition 7.10.8]{TintarevKonyv} or Skrzypczak and Tintarev \cite[Definition 1.2]{SkrzypczakTintarev}), if for every $t>0$, the set 
	 \[
	 \mathscr{O}_{t}:=\{x\in M:\:\mathrm{diam}\mathcal{O}_G^x\leq t\}
	 \]
	 is bounded.
	 Let $m(y,\rho)$ be the 
	 maximal number of mutually disjoint geodesic balls with radius $\rho$ on  $\mathcal{O}_G^y$, i.e.
	 \begin{equation}\label{myrho}
	 m(y,\rho)=\sup\left\{n\in \mathbb{N}:\exists \xi_1,\dots,\xi_n\in G\ {\rm such\ that}\ B_g(\xi_iy,\rho)\cap B_g(\xi_jy,\rho)=\emptyset,\forall i\neq j\right\},
	 \end{equation}
	 where $B_g(x,\rho)=\{z\in M:d_g(x,z)<\rho\}$ is the usual metric ball in $M$. For $\rho>0$ and $x_0\in M$ fixed, we introduce the following \textit{expansion condition} 
	 \begin{description}
	 	\item[$(\mathbf{EC})_G$]\label{D}   $m(y,\rho)\to \infty$ as $d_g(x_0,y)\to \infty$.
	 \end{description}
	 Clearly, condition $(\mathbf{EC})_G$ is independent of the choice of $x_0$.

	Now, we are in the position to state the first main result, concerning \textit{Hadamard manifold}s (i.e., simply connected, complete Riemannian manifolds with non-positive sectional curvature):
	
		\begin{theorem}\label{mainThm:Riemann1}
		Let $(M,g)$ be a $d$-dimensional Hadamard manifold, and let $G$
		be a compact connected subgroup of $\mathrm{Isom}_{g}(M)$ such that
		$\mathrm{Fix_{M}}(G)\neq\emptyset.$ Then the following statements are equivalent:
		\begin{enumerate}
			\item[{(i)}]  $G$ is coercive$;$
			\item[{(ii)}] $\ensuremath{\mathrm{Fix}_{M}(G)}$ is a singleton$;$
			\item[{(iii)}]  $(\mathbf{EC})_G$  holds.
			
		\end{enumerate}
		Moreover, from any of the above statements it follows that the embedding $W_{G}^{1,p}(M)\hookrightarrow L^q(M)$
		is compact for every $d$-admissible pair $(p,q)$. 
	\end{theorem}
	We notice that the equivalence between (i) and (ii) in Theorem \ref{mainThm:Riemann1} is proved by Skrzypczak and Tintarev \cite[Proposition 3.1]{SkrzypczakTintarev}, from which they conclude the compactness of the embedding $W_{G}^{1,p}(M)\hookrightarrow L^q(M)$
	for the admissible case (\textbf{S}); for a similar result in the case  (\textbf{MT}), see Krist\'aly \cite{KristalyJFA}. Accordingly, our purpose in Theorem \ref{mainThm:Riemann1} is to characterize their geometric properties by our expansion condition $(\mathbf{EC})_G$, by applying a careful constructive argument  based on the Rauch comparison principle, complementing also the admissible range of parameters in the Morrey-case (\textbf{M}).   
	
Our next result concerns Riemannian manifolds with \textit{bounded geometry} (i.e., complete non-compact Riemannian manifolds with Ricci curvature bounded from below having positive injectivity radius):

\begin{theorem}\label{mainThm:Riemann2}
	Let $(M,g)$ be a $d$-dimensional Riemannian manifolds with bounded geometry, and let $G$
	be a compact connected subgroup of $\mathrm{Isom}_{g}(M)$. Then the following statements are equivalent:
	\begin{enumerate}
			\item[{(i)}]  $G$ is coercive$;$
		\item[{(ii)}]  $(\mathbf{EC})_G$ holds$;$
		\item[{(iii)}]  the embedding $W_{G}^{1,p}(M)\hookrightarrow L^q(M)$	is compact for {\rm every} $d$-admissible pair $(p,q);$
		\item[{(iv)}]  the embedding $W_{G}^{1,p}(M)\hookrightarrow L^q(M)$	is compact for {\rm some} $d$-admissible pair $(p,q)$.  
	\end{enumerate}
\end{theorem}

In Theorem \ref{mainThm:Riemann2}, the equivalence between (i) and the compactness of the embedding $W_{G}^{1,p}(M)\hookrightarrow L^q(M)$
for every $d$-admissible pair $(p,q)$ in (\textbf{S}) is well known by Tintarev \cite[Theorem 7.10.12]{TintarevKonyv}; in addition, Gaczkowski, G\'{o}rka and Pons \cite{Gorka, LengyelekJFA} proved that a slightly stronger form of $(\mathbf{EC})_G$ implies (iii) in the (\textbf{S}) admissible case by using a Strauss-type argument. 
	Thus, the novelty of Theorem \ref{mainThm:Riemann2} is the equivalence of our expansion condition $(\mathbf{EC})_G$ not only with the coerciveness of $G$ but also with the validity of the compact embeddings in the full range of $d$-admissible pairs $(p,q)$. 

	Our next aim is to study similar compactness results on non-compact Finsler manifolds. We notice that in non-Riemannian Finsler settings the situation may change dramatically; indeed, there exist  non-compact Finsler--Hadamard manifolds $(M,F)$ such that the Sobolev space $W_F^{1,p}(M)$ over $(M,F)$ is \textit{not} even a vector space, see Farkas, Krist\'{a}ly and Varga \cite{FarkasKristalyVarga_Calc}, as well as Krist\'{a}ly and Rudas \cite{KristalyRudas_NATMA}. In spite of such examples, it turns out that similar compactness results to Theorems \ref{mainThm:Riemann1} \& \ref{mainThm:Riemann2} can be established on a subclass of Finsler manifolds, namely on Randers spaces with finite reversibility constant.  

	Randers spaces are specific non-reversible Finsler structures which are deduced as the solution of the Zermelo navigation problem. In fact, a 
	Randers metric shows up as a suitable perturbation of a Riemannian metric; more precisely, a Randers metric  on a manifold $M$ is a Finsler structure $F: TM \to \mathbb{R}$ defined as 
	\begin{equation} \label{Randers_metric}
		F(x,y) = \sqrt{g_x(y,y)} + \beta_x(y), \quad (x,y) \in TM,
	\end{equation} 
	where $g$ is a Riemannian metric and $\beta_x$ is a $1$-form on $M$.
	For further use, let      
	$\|\beta\|_g(x):=\sqrt{g_x^*(\beta_x,\beta_x)}$ for every $x\in M,$
	where $g^*$ is the co-metric of $g$.
	
	In order to state our result on Randers spaces, we emphasize that if $F$ is given by \eqref{Randers_metric}, then the isometry group of $(M,F)$ is a closed subgroup of the isometry group of the Riemannian manifold $(M,g)$, see Deng \cite[Proposition 7.1]{DengKonyv}. As usual, $W^{1,p}_{F,G}(M)$ stands for the subspace of $G$-invariant functions of $W^{1,p}_{F}(M)$, where $G$ is a subgroup of $\mathrm{Isom}_F(M)$, while $m_F(y,\rho)$ denotes the maximal number of mutually disjoint geodesic Finsler balls with radius $\rho$ on the orbit $\mathcal{O}_G^y$.
	

	\begin{theorem} \label{Finsler}
		Let $(M,F)$ be a $d$-dimensional Randers space endowed with the Finsler metric \eqref{Randers_metric}, such that $(M,g)$ is either a Hadamard manifold or a Riemannian manifold with bounded geometry. Let  $G$ be  a compact connected subgroup of $\mathrm{Isom}_F(M)$ such that $m_F(y,\rho)\to\infty$ as $d_{F}(x_{0},y)\to\infty$ for some $x_0\in M$ and $\rho>0$. 
		If  $\displaystyle \sup_{x\in M}\|\beta\|_g(x)<1$, then for every $d$-admissible pair $(p,q)$ the embedding $W^{1,p}_{F}(M)\hookrightarrow L^q(M)$ is continuous, while the embedding  $W^{1,p}_{F,G}(M)\hookrightarrow L^q(M)$ is compact. 
	\end{theorem}

	In fact,   assumption $\displaystyle \sup_{x\in M}\|\beta\|_g(x)<1$ in Theorem \ref{Finsler} is equivalent to the finiteness of the reversibility constant of $(M,F)$ (see Section \ref{sec:Finsler}). Furthermore, Example \ref{FinslerPelda} shows that this assumption is indispensable. Indeed,   we prove that on the Finslerian Funk model $( B^d(1), F)$, --which is a non-compact Finsler manifold of Randers-type, having infinite reversibility constant,-- the space $W^{1,p}_{F}(B^d(1))$ \textit{cannot} be continuously embedded into $L^q(B^d(1))$ for every $d$-admissible pair $(p,q)$, thus no further compact embedding can be expected.
	
	
	\newpage
	
	In the sequel, we provide an  application of Theorem \ref{Finsler} in the admissible case (\textbf{M}); we notice that applications in the admissible cases (\textbf{S}) and (\textbf{MT}) can be found in Gaczkowski, G\'{o}rka and Pons \cite{LengyelekJFA} and Krist\'aly \cite{KristalyJFA}, respectively. Accordingly, in the last part of the paper we consider the following elliptic equation on the $d$-dimensional Randers space $(M,F)$ endowed with the metric \eqref{Randers_metric}, namely 
	\begin{equation}\label{feladat}
		\begin{cases}
			-\boldsymbol{\Delta}_{F,p}u(x) = \lambda \alpha(x) h(u(x)), & x\in M,\\
			u\in W_F^{1,p}(M),
		\end{cases}\tag{$\mathcal{P}_\lambda$}
	\end{equation}
	where $\boldsymbol{\Delta}_{F,p}$ is the \textit{Finsler $p$-Laplace operator} with $p > d$, $\lambda$ is a positive parameter, $\alpha \in L^1(M) \cap L^\infty(M)$ and $h:\mathbb{R} \to \mathbb{R}$ is a continuous function. 
	For each $s\in \mathbb{R}$, put $\displaystyle H(s)=\intop_0^s h(t)\,\mathrm{d}t$, and we further assume that:
	\begin{description}
		\item[($A_1$)]\label{f1} there exists $s_0>0$ such that, $H(s)>0 \ \forall s\in (0,s_0]$;
		\item[($A_2$)]\label{f2} there exist $C>0$ and $1<w<p$ such that $|h(s)|\leq C(1+|s|^{w-1}),\ \forall s\in \mathbb{R}$;
		\item[($A_3$)]\label{f3} there exists $q>p$ such that $$\limsup_{s\to 0}\frac{H(s)}{|s|^q}<\infty.$$
	\end{description}

	
	\begin{theorem}\label{alkalmazas} 
		Let $(M,F)$ be a $d$-dimensional Randers space endowed with the Finsler metric \eqref{Randers_metric} such that $\sup_{x\in M}\|\beta\|_g(x)<1$ and $g$ is a Riemannian metric where $(M,g)$ is a Hadamard manifold with sectional curvature bounded above by $-\kappa^2$, $\kappa>0$. Let $G$ be a compact connected subgroup of $\mathrm{Isom}_F(M)$ such that $\mathrm{Fix}_M(G) = \{x_0\}$ for some $x_0\in M$. Let $h:\mathbb{R} \to \mathbb{R}$ be a continuous function verifying \hyperref[f1]{$(A_1)$} -- \hyperref[f3]{$(A_3)$}, and $\alpha\in L^1(M) \cap L^\infty(M)$ be a non-zero, non-negative function which depends on $d_F(x_0, \cdot)$ and  satisfies $\displaystyle \sup_{R>0}\underset{d_{F}(x_{0},x)\leq R}{\mathrm{essinf}}\alpha(x) > 0.$ Then there exists an open interval $\Lambda \subset [0,\lambda^*]$ and a number $\mu > 0$ such that for every $\lambda\in \Lambda$, problem \eqref{feladat} admits at least three solutions in $W^{1,p}_{F,G}(M)$ having $W^{1,p}_{F}(M)$-norms less than $\mu$.
	\end{theorem}
	
	The proof of Theorem \ref{alkalmazas} is based on the compact embedding from Theorem \ref{Finsler} combined with variational arguments.

	The organization of the paper is the following. After presenting some preliminary results in Riemannian geometry (see Section \ref{sec:Riemann}), Sections \ref{3-sect} and \ref{4-sect} are devoted to the proof of Theorems \ref{mainThm:Riemann1} \& \ref{mainThm:Riemann2}, respectively. Section \ref{sec:Finsler} contains preliminaries on Randers spaces and the proof of Theorem \ref{Finsler}, together with Example  \ref{FinslerPelda}, emphasizing the sharpness of Theorem \ref{Finsler}. Finally, in the last part of Section \ref{sec:Finsler}, we present  the proof of Theorem \ref{alkalmazas}.

	\section{Preliminaries} \label{sec:Riemann}

	Let $(M,g)$ be a complete non-compact Riemannian manifold with $\mathrm{dim}M=d$.  
	Let $T_xM$ be the tangent space at $x \in M$, $\displaystyle TM =
	\bigcup_{x\in M}T_xM$ be the tangent bundle, and $d_g : M \times M
	\to [0, +\infty)$ be the distance function associated to the
	Riemannian metric $g$. Let $B_g(x, \rho) = \{y \in M : d_g(x, y) <
	\rho \}$ be the open metric ball with center $x$ and radius $\rho> 0$; if $dv_g$ is the canonical volume element on $(M, g)$, the
	volume of a bounded open set $\Omega \subset M$ is 
	$\mathrm{Vol}_g(\Omega) = \displaystyle\int_{\Omega} {\rm d}v_{g}= \mathcal H^d(\Omega)$.
	If ${\text d}\sigma_g$ denotes
	the $(d-1)$-dimensional Riemannian measure induced on $\partial \Omega$
	by $g$, then
	$$\mathrm{Area}_g(\partial \Omega)=\displaystyle\int_{\partial \Omega} {\text d}\sigma_g=\mathcal H^{d-1}(\partial \Omega)$$ stands for the area of $\partial \Omega$ with
	respect to the metric $g$. Hereafter, $\mathcal H^l$ denotes the $l$-dimensional Hausdorff measure. 
	
	Let $p>1.$ The norm of $L^p(M)$ is given by
	$$\|u\|_{L^p(M)}=\left(\displaystyle\int_M
	|u|^p{\rm d}v_{g}\right)^{1/p}.$$ Let $u:M\to \mathbb R$ be a function of
	class $C^1.$ If $(x^i)$ denotes the local coordinate system on a
	coordinate neighbourhood of $x\in M$, and the local components of
	the differential of $u$ are denoted by 
	$u_i=\frac{\partial	u}{\partial x_i}$, then the local components of the gradient  $\nabla_g u$ are
	$u^i=g^{ij}u_j$. Here, $g^{ij}$ are the local components of
	$g^{-1}=(g_{ij})^{-1}$. In particular, for every $x_0\in M$ one has
	the eikonal equation
	\begin{equation}\label{dist-gradient}
		|\nabla_g d_g(x_0,\cdot)|=1\ {\rm a.e. \ on}\ M.
	\end{equation}
	
	When no confusion arises, if $X,Y\in T_x M$, we simply write $|X|$ and
	$\langle X,Y\rangle$ instead of the norm $|X|_x$ and inner product
	$g_x(X,Y)=\langle X,Y\rangle_x$, respectively. 
	
	The $L^p(M)$ norm of $\nabla_g u: M \to TM$ is given by
	$$\|\nabla_g u\|_{L^p(M)}=\left(\displaystyle\int_M |\nabla_gu|^p{\rm d}v_g\right)^\frac{1}{p}.$$  
	The space $W^{1,p}_g(M)$ is the completion of $C_0^\infty(M)$ with respect to the norm
	$$\|u\|^p_{W^{1,p}_g(M)}={\|u\|_{L^p(M)}^p+\|\nabla_g u\|_{L^p(M)}^p}.$$
	
	For any $c\leq 0$, let
	$$V_{c,d}(\rho) = d \omega_d \int_0^\rho {\bf s}_c(t)^{d-1}{\rm d}t$$
	be the volume of the ball with radius $\rho>0$ in the $d$-dimensional
	space form (i.e., either the hyperbolic space with sectional
	curvature $c$ when $c<0$, or the Euclidean space when $c=0$), where
	$${\bf s}_{c}(t)=\left\{
	\begin{array}{lll}
		t
		& \hbox{if} &  {c}=0, \\
		\frac{\sinh(\sqrt{-c}t)}{\sqrt{-c}} & \hbox{if} & {c}<0,
	\end{array}\right.$$
	and $\omega_d$ is the volume of the
	unit $d$-dimensional Euclidean ball. Note that for every $x\in M$,
	we have
	\begin{equation}  \label{volume-comp-nullaban}
		\lim_{\rho\to 0^+}\frac{\mathrm{Vol}_g(B_g(x,\rho))}{V_{c,d}(\rho)} = 1.
	\end{equation}
	The notation ${\bf K}\leq {c}$ means that the sectional curvature is
	bounded from above by ${c}$ at any point and  direction.
	The Bishop-Gromov volume comparison principle states that if  $(M,g)$ be a $d$-di\-men\-sional Hadamard manifold with  
	${\bf	K}\leq c\leq 0$ and $x\in M$ fixed, then the function 
	$$\rho\mapsto \frac{\mathrm{Vol}_g(B_g(x,\rho))}{V_{c,d}(\rho)}, \ \rho>0$$ 
	is non-decreasing;
	in particular, from {\rm \textrm{(\ref{volume-comp-nullaban})}} one has
	\begin{equation}  \label{volume-comp-altalanos-0}
		{\mathrm{Vol}_g(B_g(x,\rho))}\geq V_{c,d}(\rho)\ {\rm for\ all}\ \rho>0.
	\end{equation}
	If equality holds in (\ref{volume-comp-altalanos-0}) for all $x\in M$ and $\rho>0$, then $\mathbf{K}\equiv c$; for further details, see Shen \cite{ShenVolume}.

In a similar way, if the Ricci curvature of $(M,g)$ is bounded from below by $(n-1)c$ (with $c\leq0$), then 
$$\rho\mapsto \frac{\mathrm{Vol}_g(B_g(x,\rho))}{V_{c,d}(\rho)}, \ \rho>0$$ 
is non-increasing;
moreover, by {\rm \textrm{(\ref{volume-comp-nullaban})}} one has
\begin{equation}  \label{volume-comp-altalanos-1}
{\mathrm{Vol}_g(B_g(x,\rho))}\leq V_{c,d}(\rho)\ {\rm for\ all}\ \rho>0.
\end{equation}

	Let $G$ be a compact connected subgroup of $\mathrm{Isom}_g(M)$, and let $\mathcal{O}_G^x=\{\xi x: \xi \in G\}$ be the orbit of the element $x\in M$.
	The action of $G$ on $W^{1,p}_g(M)$ is defined by
	\begin{equation}\label{action-of-the-group}
		(\xi u)(x)=u(\xi^{-1}x) \ \ {\rm for\ all}\  x\in M,\ \xi \in G,\ u\in W^{1,p}_g(M),
	\end{equation}
	where
	$\xi^{-1}:M\to M$ is the inverse of the isometry $\xi$. We say that a continuous action of a group $G$ on a complete Riemannian manifold $M$ is \textit{coercive} (see Tintarev \cite[Definition 7.10.8]{TintarevKonyv} or Skrzypczak and Tintarev \cite[Definition 1.2]{SkrzypczakTintarev}) if for every $t>0$, the set 
	\[
	\mathscr{O}_{t}=\{x\in M:\:\mathrm{diam}\mathcal{O}_G^x\leq t\}
	\]
	is bounded. Let
	$$W^{1,p}_{G}(M)=\{u\in W^{1,p}_g(M): \xi u=u\
	{\rm for\ all}\ \xi \in G\}$$ 
	be the subspace of $G$-invariant functions of $W^{1,p}_g(M)$.
	
	Let $\mathscr{C}(M)$ be the space of continuous functions $u: M \to [0,\infty)$ having compact support $D \subset
	M$, where $D$ is  smooth enough, $u$ being of class
	$C^2$ in $D$ and having only non-degenerate critical points in $D$.
	Based on classical Morse theory and density arguments, in the sequel we shall consider test functions $u \in \mathscr{C}(M)$ in order to handle
	generic Sobolev inequalities.
	
	Let $u\in \mathscr{C}(M)$  and $\Omega\subset {\rm supp}(u)\subset M$ be an open set. Similarly to Druet, Hebey, and Vaugon \cite{DruetHebeyVaugon}, we may associate to the restriction of $u$ to $\Omega$, namely $u|_\Omega$, 
	its
	{\it Euclidean}  {\it rearrangement} function $u^*:B_e(0,R_\Omega)\to [0,\infty)$, which is
	radially symmetric, non-increasing in $|x|$, and for every $t\geq \inf_\Omega u$ is defined by
	\begin{equation}\label{vol-egyenloseg}
		{\rm Vol}_e(\{x\in B_e(0,R_\Omega):u^*(x)>t\})={\rm Vol}_g(\{x\in \Omega:u(x)>t\});
	\end{equation}
	here, Vol$_e$ denotes the usual $d$-dimensional Euclidean volume and $R_\Omega>0$ is chosen such that ${\rm Vol}_g(\Omega)={\rm Vol}_e(B_e(0,R_\Omega))=\omega_dR_\Omega^d$. 
	In the sequel, we state the most important properties of this rearrangement which are crucial in the proof of Theorem \ref{mainThm:Riemann1}; the proof relies on suitable application of the co-area formula combined with the weak form of the isoperimetric inequality on Hadamard manifolds (for a similar proof, see Druet, Hebey, and Vaugon \cite{DruetHebeyVaugon}, and Krist\'aly \cite{KristalyMorreyPotencial}).

	\begin{lemma}\label{PolyaSzego}
		Let $(M,g)$ be a $d(\geq 2)-$dimensional Hadamard manifold. 
		Let  $u \in \mathscr{C}(M)$ be a non-zero function, $\Omega\subset {\rm supp}(u)\subset M$ be an open set,  and $u^*:B_e(0,R_\Omega) \to [0,\infty)$ its Euclidean
		rearrangement function.  Then the following properties hold:
		\begin{itemize}
			\item[{\rm (i)}] {\rm Norm-preservation:} for every $q\in (0,\infty],$
			$\|u\|_{L^q(\Omega)}=\|u^*\|_{L^q(B_e(0,R_\Omega))};$

			\item[{\rm (ii)}] {\rm P\'olya-Szeg\H{o} inequality:} for every $p>1,$ 
			\begin{equation}\label{PSz-sharp}
				\|\nabla_g u\|_{L^p(\Omega)}\geq \frac{C(d)}{d \omega_d^\frac{1}{d}}\|\nabla u^*\|_{L^p(B_e(0,R_\Omega))},
			\end{equation}   
			where $C(d)>0$ is the Croke-constant $($see Croke \cite{Croke}$),$ i.e., $C(2)=1$ and 
			$$C(d)={(d\omega_d)^{1-\frac{1}{d}}}\left((d-1)\omega_{d-1}\int_0^\frac{\pi}{2}\cos^\frac{d}{d-2}
			(t)\sin^{d-2} (t)\,\mathrm{d}t\right)^{\frac{2}{d}-1},\ \ d\geq 3.$$
		\end{itemize}
	\end{lemma}
	
		We conclude this section with the following Rellich--Kondrachov-type embedding, an expected result based on Aubin \cite[Chapter 2]{Aubin}; nevertheless, for convenience, we propose here an alternative proof which is needed both in Theorems \ref{mainThm:Riemann1} and \ref{mainThm:Riemann2}.

	\begin{lemma}\label{RK-lemma} 
		Let $(M,g)$ be a $d$-dimensional complete Riemannian manifold. If $R>0$, then the embedding $W^{1,p}_g(B_g(y,R))\embd L^q({B_g(y,R)})$ is compact for every $y\in M$ and every $d$-admissible pair $(p,q)$.
	\end{lemma}
	\begin{proof} 
		Since $\overline{B_g(y,R)} \subset M$ is compact (due to Hopf-Rinow theorem), the Ricci curvature is bounded from below, see Bishop and Crittenden \cite[p. 166]{BishopCrittenden} and the injectivity radius is positive on $\overline{B_g(y,R)}$, see Klingenberg \cite[Proposition 2.1.10]{Klingenberg} or Bao, Chern, and Shen \cite[Chapter 8]{BaoChernShen}.  
		
		Thus, we are in the position to use Hebey \cite[Theorem 1.2]{Hebey-konyv1}; therefore, for every $\varepsilon>0$ there exists a harmonic radius $r_H>0$, such that for every $z\in \overline{B_g(y,R)}$ one can find a harmonic coordinate chart $\varphi_z: B_g(z,r_H)\to \mathbb{R}^d$ such that $\varphi_z(z)=0$ and the components $(g_{jl})$ of $g$ in this chart satisfy 
		\begin{equation}\label{csak-egy-y-ra}
		\displaystyle \frac{1}{1+\varepsilon}\delta_{jl}\le g_{jl}\le (1+\varepsilon) \delta_{jl}
		\end{equation} as bilinear forms. 
		Therefore, it follows that 
		\begin{equation}\label{egyenletolenseg}
		\frac{1}{\sqrt{1+\varepsilon}}d_g(z,x) \le |\varphi_z(x)| \leq \sqrt{1+\varepsilon}d_g(z,x), \ \text{for all } x \in B_g(z, r_H).
		\end{equation}
		
		Now let $0 < \rho < r_H$. Since $\overline{B_{g}(y,R)}$ is compact, there exists $L\in \mathbb{N}$ and $z_1,\dots,z_L \in \overline{B_{g}(y,R)}$ such that 
		$\displaystyle \overline{B_g(y,R)} \subseteq \bigcup_{j=1}^L B_g(z_j,\rho)$.
		For every $z_j \in \overline{B(y,R)}, j=\overline{1,L}$, denote by 
		$$U_{z_j} \coloneqq B_g(z_j,\rho) \cap B_g(y,R) \ \text{ and } \
		\Omega_{z_j} \coloneqq \varphi_{z_j}\left( U_{z_j} \right) \subset \mathbb{R}^d ,$$ 
		thus $\left\{ U_{z_j} \right \}_{j=\overline{1,L}}$ is a finite covering of $B_g(y,R)$.
		
		First observe that for any $j \in \{1,\dots,L\}$ and  $u\in W_g^{1,p}(B_g(y,R))$, on account of \eqref{egyenletolenseg}, we have that 
		\begin{equation}\label{egyik1}
		\int_{U_{z_j}} |\nabla_{g}u|^p+|u|^p {\rm d}v_{g} \geq 
		\left(\frac{1}{\sqrt{1+\varepsilon}}\right)^{d+p}\left(\int_{\Omega_{z_j}} 
		|\nabla(u\circ \varphi_{z_j}^{-1})|^p+|u\circ \varphi_{z_j}^{-1}|^p \mathrm{d}x\right).
		\end{equation}

		We first focus on the (\textbf{S}) admissible case. Observe that 	\begin{equation}\label{masik}
			\int_{U_{z_j}}|u|^q{\rm d}v_{g} \le (1+\varepsilon)^\frac{d}{2}\int_{\Omega_{z_j}}|u\circ \varphi_{z_j}^{-1}|^q\mathrm{d}x.
		\end{equation}
	
Now, by the  euclidean Sobolev inequality (see Brezis \cite[Corollary 9.14]{Brezis}), for every $j\in \{1,\dots,L\}$ there exists a constant $C_{S,j}$ such that  \begin{equation}\label{eucSobS}\left(\int_{\Omega_{z_j}}|u\circ \varphi_{z_j}^{-1}|^q\mathrm{d}x\right)^\frac{1}{q}\leq C_{S,j} \left(\int_{\Omega_{z_j}} |\nabla(u\circ \varphi_{z_j}^{-1})|^p+|u\circ \varphi_{z_j}^{-1}|^p\mathrm{d}x\right)^{\frac{1}{p}}.
\end{equation}
	Therefore, by \eqref{egyik1}, \eqref{masik} and \eqref{eucSobS} we have that 
	\begin{align}
		\|u\|_{L^q(B_g(y,R))} &\leq \sum_{j=1}^{L} \|u\|_{L^q(U_{z_j})}\leq (1+\varepsilon)^{\frac{d}{2q}}\sum_{j=1}^{L}\|u\circ \varphi^{-1}_{z_j}\|_{L^q(\Omega_{z_j})}\nonumber\\ &\leq (1+\varepsilon)^{\frac{d}{2q}}\sum_{j=1}^{L} C_{S,j}\|u\circ \varphi^{-1}_{z_j}\|_{W^{1,p}(\Omega_{z_j})} \nonumber \leq (1+\varepsilon)^{\frac{dp+dq+pq}{2pq}}\sum_{j=1}^{L}C_{S,j} \|u\|_{W^{1,p}_g(U_{z_j})}\nonumber \\&\leq (1+\varepsilon)^{\frac{dp+dq+pq}{2pq}}\sum_{j=1}^{L} C_{S,j}\cdot \|u\|_{W^{1,p}_g(B_g(y,R))},\label{SSobLocal}
	\end{align}
which proves the validity of the continuous Sobolev embedding $W_{g}^{1,p}(B_{g}(y,R))\hookrightarrow L^{q}(B_{g}(y,R))$ in the $(\textbf{S})$ case. 
Now we prove that the previous embedding is compact. To do this, let $\{u_n\}_n$ be a bounded sequence in $W^{1,p}_g(B_g(y,R))$, and denote $\tilde{u}_n^j=u_n|_{U_{z_j}}$ for every $j \in \{1, \dots , L\}$. Using \eqref{egyik1}, we have that for every $j$, the sequence $\tilde{u}_n^j=u_n\circ \varphi_{z_j}^{-1}$ is bounded in $W^{1,p}(\Omega_{z_j})$.  By the Rellich-Kondrachov theorem  one gets that there exists a subsequence of $\{\tilde{u}_n^j\}_n$ which is a Cauchy sequence in $L^q(\Omega_{z_j})$. Let $\{u_m\}_m$ be a subsequence of $\{u_n\}_n$ such that for any $j$, $\{\tilde{u}_m^j\}_m$ is a Cauchy sequence in $L^q(\Omega_{z_j})$. Thus, applying \eqref{masik}, for any $m_1,m_2$ we have that $$\|u_{m_1}-u_{m_2}\|_{L^q(B_g(y,R))}\leq \sum_{j=1}^{L}\|u^j_{m_1}-u^j_{m_2}\|_{L^q(U_{z_j})}\leq(1+\varepsilon)^{\frac{d}{2q}} \sum_{j=1}^{L}\|\tilde{u}_{m_1}^j-\tilde{u}_{m_2}^j\|_{L^q(\Omega_{z_j})},$$ hence $\{u_m\}_m$ is a Cauchy sequence in $L^q(B_g(y,R))$, which proves the claim.

One can prove the (\textbf{MT}) admissible case analogously, replacing \eqref{eucSobS} with the euclidean Sobolev inequality when $p = d$.

	Finally, in the (\textbf{M}) case, we have that	
		
		\begin{equation}\label{supNorm1}
			\sup_{x\in\overline{B_{g}(y,R)}}|u(x)|=\max_{j=\overline{1,L}}\|u\|_{C^0(\overline{U_{z_j}})} = 
			\max_{j=\overline{1,L}}\|u\circ \varphi_{z_j}^{-1}\|_{C^0(\overline{\Omega_{z_j}})}.
		\end{equation}
		Again, by Brezis \cite[Corollary 9.14]{Brezis}, for each $j\in \{1,\dots,L\}$ there exists a constant $C_{0,j}$ such that 
		$$\|u\circ \varphi_{z_j}^{-1}\|_{C^0(\overline{\Omega_{z_j}})} \leq C_{0,j} \cdot \|u\circ \varphi_{z_j}^{-1}\|_{W^{1,p}(\Omega_{z_j})},$$
		thus this inequality together with \eqref{egyik1} and \eqref{supNorm1} yields that
		\begin{align}
			\sup_{x\in\overline{B_{g}(y,R)}}|u(x)|&\leq \max_{j=\overline{1,L}} C_{0,j} \|u\circ \varphi_{z_j}^{-1}\|_{W^{1,p}(\Omega_{z_j})} \nonumber \\
			&\leq \max_{j=\overline{1,L}}C_{0,j} (1+\varepsilon)^\frac{d+p}{2p} \|u\|_{W_g^{1,p}(U_{z_j})} \nonumber \\
			&\leq \max_{j=\overline{1,L}}C_{0,j} \cdot (1+\varepsilon)^\frac{d+p}{2p} \|u\|_{W_g^{1,p}(B_g(y,R))},\label{supNorm2}
		\end{align}
		which proves again that the continuous embedding holds.  
		Now we prove that this injection is compact. To do this, consider a bounded set $A\subset W_g^{1,p}(B_g(y,R))$ , i.e. there exists $M>0$ such that 
		$\|u\|^p_{W_g^{1,p}(B_g(y,R))}\leq M \ \text{ for all } u \in A.$ 
		From the previous inequality and \eqref{supNorm2} it follows that there exists $C_2>0$ such that $\|u\|_{C^0(\overline{B_g(y,R)})}\leq M C_2$ for all $u \in A.$  Thus by Ascoli's Theorem (see  Aubin \cite[Theorem 3.15]{Aubin}), we get that $A$ is precompact in $C^0(\overline{B_g(y,R)})$, which concludes the proof.	
\end{proof}

	\section{Proof of Theorem \ref{mainThm:Riemann1}}\label{3-sect}

 $\underline{({\rm i}) \Leftrightarrow ({\rm ii})}$ This equivalence can be found in Skrzypczak and Tintarev \cite[Proposition 3.1]{SkrzypczakTintarev}.	

$\underline{({\rm ii})\Rightarrow({\rm iii})}$  Without loss of any generality, it is enough to prove that $m(\gamma(t),\rho)\to \infty$ as $t\to \infty$ for every unit speed geodesic $\gamma:[0,\infty)\to M$ emanating from $x_0=\gamma(0),$ i.e., $\gamma(t)=\exp_{x_0}(ty)$ for some $y\in T_{x_0}M$ with $|y|_{g_{x_0}}=1,$ where $g_{x_0}$ and $|\cdot|_{g_{x_0}}$ denote  the inner product and norm on $T_{x_0}M$ induced by the metric $g$. 

We notice that $\mathcal{O}_G^{\gamma(t)}$ contains infinitely many elements for every $t>0$. Indeed,  $\mathcal{O}_G^{\gamma(t)}$ is a connected submanifold of $M$ whose dimension is at least 1; if its dimension would be 0 for some $t_0>0$, by connectedness,  $\mathcal{O}_G^{\gamma(t_0)}$ would be a singleton, i.e.,  $$\gamma(t_0)\in \mathrm{Fix}_M(G)=\{x_0\}=\{\gamma(0)\},$$ which is a contradiction. 
Therefore, card$\mathcal{O}_G^{\gamma(t)}=+\infty$ for every $t>0$.

If for a fixed $t_0>0$, we choose different elements $\xi_i\in G$, $i\in \mathbb N$ such that $\xi_i \gamma(t_0) \in \mathcal{O}_G^{\gamma(t_0)}$, then we also have $(\xi_i\circ \gamma)(t)=\xi_i \gamma(t) \in \mathcal{O}_G^{\gamma(t)}$ for every $i\in \mathbb N$ and $t>0$; the latter statement immediately follows from the fact that $\xi_i\in G$, $i\in \mathbb N$ are isometries, thus $t\mapsto (\xi_i\circ \gamma)(t)$, are also geodesics of unit speed emanating from $x_0.$  

Let us transplant the geodesic balls $B_g(\xi_i \gamma(t),\rho)\subset M$, $i\in \mathbb N$, into the tangent space $T_{x_0}M$ by the exponential map $\exp_{x_0}$, i.e., $\exp_{x_0}^{-1}(B_g(\xi_i \gamma(t) ,\rho))\subset T_{x_0}M$, $i\in \mathbb N$.

We claim that 
\begin{equation}\label{hasonlitas}
	\exp_{x_0}^{-1}(B_g(\xi_i \gamma(t),\rho))\subset B^{x_0}_\rho(\exp_{x_0}^{-1}(\xi_i \gamma(t))) \eqqcolon B_i^t(\rho),\ i\in \mathbb N,
\end{equation}
where $B^{x_0}_\rho(v)=\{z\in T_{x_0}M:|v-z|_{g_{x_0}}<\rho\}\subset T_{x_0}M$ for any $v\in T_{x_0}M.$ 

To see this, let $i\in \mathbb N$ and $t \in [0, \infty)$ be arbitrarily fixed. Take an element  $z\in \exp_{x_0}^{-1}(B_g(\xi_i \gamma(t),\rho))$, thus $\tilde z \coloneqq \exp_{x_0}(z)\in B_g(\xi_i \gamma(t) ,\rho)$. If $z = \exp_{x_0}^{-1}(\xi_i \gamma(t))$, we have nothing to prove. Otherwise, consider the geodesic triangle uniquely determined by the points $x_0$, $\xi_i \gamma(t)$ and $\tilde z$, respectively. Since $(M,g)$ is a Hadamard manifold, the Rauch comparison principle (see e.g. do Carmo \cite[Proposition 2.5, p. 218]{doCarmo}) implies that 
$$|\exp_{x_0}^{-1}(\xi_i \gamma(t))- z|_{g_{x_0}}=|\exp_{x_0}^{-1}(\xi_i \gamma(t))-\exp_{x_0}^{-1}(\tilde z)|_{g_{x_0}}\leq d_g(\xi_i \gamma(t),\tilde z)<\rho,$$
which concludes the proof of \eqref{hasonlitas}. 

Since the geodesics $\xi_i \circ \gamma$ are mutually different for any $i\in \mathbb N$,  the angle between any two vectors  \mbox{$\exp_{x_0}^{-1}(\xi_i \gamma(t))\subset T_{x_0}M$} are positive and it does \textit{not} depend on the value of $t>0$. Let $\alpha_{ij}\in (0,\pi]$ be the angle between $v_i \coloneqq \exp_{x_0}^{-1}(\xi_i \gamma(t))$ and $v_j \coloneqq \exp_{x_0}^{-1}(\xi_j \gamma(t))$, $i\neq j$.

Geometrically, the semilines $\tau\mapsto \tau v_i\subset T_{x_0}M$, $\tau>0$, move away in $T_{x_0}M$ from each other, independently of $t>0.$ Accordingly, 
it turns out that larger values of $t>0$ imply more mutually disjoint balls of the form $B_i^t(\rho)$. More precisely, if we define  
$$\tilde m(t,\rho)=\sup\left\{n\in \mathbb{N}: B_k^t(\rho)\cap B_l^t(\rho)=\emptyset,\forall k\neq l\ {\rm with}\ k,l\in \{1,\dots,n\}\right\},$$
we claim that $\tilde m(t,\rho)\to \infty$ as $t\to \infty$. To prove this, for every $n\geq 2$,  let 
$$t_n \coloneqq \max\left\{\frac{\rho}{\sin \left(\frac{\alpha_{ij}}{2} \right)}:i,j\in \{1,\dots,n\},i\neq j\right\}.$$
Let $t_1=0$. By the latter definition, it turns out that $\tilde m(t,\rho)\geq n$ whenever $t\geq t_n.$ Let us observe that the sequence $\{t_{n}\}_n$ is non-decreasing and $\displaystyle \lim_{n\to \infty}t_n=+\infty$. The former statement is trivial, while the limit follows from the fact that the sequence of  $w_i:=\frac{v_i}{|v_i|_{g_{x_0}}}$, $i\in \mathbb N$ (belonging to the unit sphere of $T_{x_0}M$ with center $0\in T_{x_0}M$) has a convergent subsequence, say $\{w_{i_k}\}_k$; in particular, the sequence of angles $\{\alpha_{i_ki_{k+1}}\}_k$ converges to 0, which implies the validity of the required limit. 

Now, let $\{t_{n_k}\}_k$ be a strictly increasing subsequence of $\{t_n\}_n$ with $t_{n_1}=t_1=0$, and let $f:[0,\infty)\to [0,\infty)$ be defined by 
$$f(s)=t_{n_k}+(s-k)(t_{n_{k+1}}-t_{n_k}),$$ 
for every $s\in [k,k+1)$, $k\in \mathbb N$. It is clear that $f$ is strictly increasing and  $\displaystyle \lim_{s\to \infty}f^{-1}(s)=+\infty.$
By the above construction, for every $t>0$, there exists a unique $k\in \mathbb N$ such that $t_{n_k}\leq t<t_{n_{k+1}}$. 

In particular, it follows that  $k=f^{-1}(t_{n_k})\leq f^{-1}(t)<f^{-1}(t_{n_{k+1}})=k+1$, thus
$$f^{-1}(t)-1 < k \leq n_k \leq \tilde m(t,\rho).$$
The above relation immediately implies that $\tilde m(t,\rho)\to \infty$ as $t\to \infty$.

On the other hand, by (\ref{hasonlitas}) and the fact that $\exp_{x_0}$ is a diffeomorphism, it turns out that 
$$B_g(\xi_i \gamma(t),\rho)\cap B_g(\xi_j \gamma(t),\rho)=\emptyset,\ \forall i\neq j\ {\rm with}\ i,j\in \{1,\dots,\tilde m(t,\rho)\}.$$
Therefore, we have that 
\begin{equation}\label{m-tilde-m}
	m(\gamma(t),\rho)\geq \tilde m(t,\rho),
\end{equation} and the aforementioned limit concludes the proof.

 $\underline{({\rm iii})\Rightarrow ({\rm ii})}$ Let us assume that the set $\mathrm{Fix}_G(M)$ is not a singleton, i.e. there exists $x_0,x_1\in \mathrm{Fix}_G(M)$ such that $\delta:=d_g(x_0,x_1)>0$. Since $M$ is a Hadamard manifold, there exists a unique minimal geodesic   $\gamma:\mathbb R\to M$, parametrized by arc-length, and passing throughout the points $x_0$ and $x_1$.  Let $x_2\in {\rm Im}\gamma\setminus\{x_0\}$ be such that $d_g(x_1,x_2)=\delta$ and  $t_0<t_1<t_2$ with $x_i=\gamma(t_i)$, $i\in \{0,1,2\}$.  Fix an arbitrary element $\xi \in G$; in particular, $t\mapsto \widetilde{\gamma}(t) \coloneqq (\xi \circ \gamma)(t)$ is also a geodesic. 

It is clear that $\widetilde{\gamma}(t_2)=\xi x_2$ and due to the fact that  $x_0,\,x_1\in \mathrm{Fix}_G(M)$, it turns out that $\widetilde{\gamma}(t_i)=\xi x_i=x_i$, $i\in \{0,1\}$. Therefore, by the uniqueness of the geodesic between $x_0$ and $x_1$, it follows that $\tilde \gamma(t)=\gamma(t)$ for every $t\in [t_0,t_1]$. Since Riemannian manifolds are non-branching spaces, it follows in fact that  $\widetilde\gamma \equiv {\gamma}$, thus $\xi x_2=x_2$; by the arbitrariness of $\xi\in G$ we obtain that $x_2\in \mathrm{Fix}_G(M)$ and $d_g(x_0,x_2)=d_g(x_0,x_1)+d_g(x_1,x_2)=2\delta$. By repeating this argument, one can construct a sequence of point $\{x_n\}_n\subset M$ such that $x_n\in \mathrm{Fix}_G(M)$ and $d_g(x_0,x_n)=n\delta$, $n\in \mathbb N$. In particular, $d_g(x_0,x_n)\to \infty$ as $n\to \infty$ and since $x_n\in \mathrm{Fix}_G(M)$ for every $n\in \mathbb N$, it follows that  $m(x_n,\rho)=1$, which is a contradiction.

\underline{$({\rm ii}) \Rightarrow$ compact embeddings.} First of all, the compactness of  embeddings $W_{G}^{1,p}(M)\hookrightarrow L^q(M)$  in the admissible cases  (\textbf{S}) and (\textbf{MT}) follow by Skrzypczak and Tintarev \cite{SkrzypczakTintarev}. It remains to consider the  admissible case (\textbf{M}), i.e. to prove the compactness of $W_{G}^{1,p}(M)\hookrightarrow L^\infty(M)$ whenever $p>d.$  

To complete this, we first claim that for every $\rho>0$ fixed, one has 
\begin{equation}\label{M-elso}
{\displaystyle \inf_{y\in M}{\displaystyle S(y,\rho)^{-1}>0}},
\end{equation}
where $S(y,\rho)$ is the embedding constant defined by the embedding $W^{1,p}_g(B_g(y,\rho))\embd C^0(\overline{B_g(y,\rho)})$, see 	Lemma \ref{RK-lemma}. It is clear that $\displaystyle S(y,\rho)>0$ can be considered for non-negative and non-zero functions. 

To prove \eqref{M-elso}, for $y\in M$ arbitrarily fixed,  let $u\in W^{1,p}_g(B_{g}(y,\rho))\setminus \{0\}$ be non-negative. By Lemma \ref{PolyaSzego}/(ii) it turns out that 
\begin{equation}\label{localPSZ}
\int_{B_{g}(y,\rho)} |\nabla_{g} u|^p\,{\rm d}v_{g} \geq \frac{C(d)}{d\omega_d^\frac{1}{d}}\int_{B_{e}(0,\tilde{\rho}_{y})}|\nabla u^*|^p \, \mathrm{d}x,
\end{equation}
where $u^*:B_{e}(0,\tilde{\rho}_{y})\to [0,\infty)$ denotes the Euclidean rearrangement of $u$; in particular,  
we have 
\begin{equation}\label{eq:gombok_terfogata_egyenlo.}
\mathrm{Vol}_{g}(B_{g}(y,\rho))=\mathrm{Vol}_{e}(B_{e}(0,\tilde{\rho}_{y}))=\omega_{d}\cdot\tilde{\rho}_{y}^{d},
\end{equation}
and 
\begin{equation} \label{eq:fontos}
\sup_{x\in\overline{B_{g}(y,\rho)}}|u(x)|=\sup_{x\in\overline{B_{e}(0,\tilde{\rho}_{y})}}|u^{*}(x)|=u^*(0).
\end{equation}
On the other hand, by the Bishop-Gromov theorem (see \eqref{volume-comp-altalanos-0}) together
with (\ref{eq:gombok_terfogata_egyenlo.}), one can see that $\rho\leq\tilde{\rho}_{y}.$
Thus $B_{e}(0,\rho)\subseteq B_{e}(0,\tilde{\rho}_{y}),$ and $W^{1,p}(B_{e}(0,\tilde{\rho}_{y}))\subseteq W^{1,p}(B_{e}(0,\rho))$. Accordingly,   
\begin{align*}S(y,\rho)^{-1}&=\inf_{u\in W^{1,p}(B_g(y,\rho))}\frac{\displaystyle \left(\int_{B_g(y,\rho)}|\nabla_gu|^p{\rm d}v_{g}+\int_{B_g(y,\rho)}|u|^p{\rm d}v_{g}\right)^\frac{1}{p}}{\displaystyle \sup_{x\in \overline{B_g(y,\rho)}}|u(x)|}\\ &{\geq} \frac{C(d)}{d\omega_{d}^\frac{1}{d}} \inf_{u^*\in W^{1,p}(B_e(0,\tilde\rho_y))}\frac{{\displaystyle \left(\int_{B_{e}(0,\tilde{\rho}_{y})}|\nabla u^*|^{p}\mathrm{d}x+\int_{B_{e}(0,\tilde{\rho}_{y})}|u^*|^{p}\mathrm{d}x\right)^{\frac{1}{p}}}}{\displaystyle \sup_{x\in \overline{B_e(0,\tilde\rho_y)}}|u^*(x)|}\\  &\geq\frac{C(d)}{d\omega_{d}^\frac{1}{d}} \inf_{u^*\in W^{1,p}(B_e(0,\rho))}\frac{\|u^*\|_{W^{1,p}(B_e(0,\rho))}}{u^*(0)}=\frac{C(d)}{d\omega_{d}^\frac{1}{d}} \inf_{u^*\in W^{1,p}(B_e(0,\rho))}\frac{\|u^*\|_{W^{1,p}(B_e(0,\rho))}}{\displaystyle \sup_{x\in \overline{B_e(0,\rho)}}|u^*(x)|}>0.
\end{align*}
Since the latter value does \textit{not} depend on $y\in M$,  we conclude the proof of \eqref{M-elso}.

Now, let  $\{u_n\}_n\subset W^{1,p}_G(M)$ be a bounded sequence and $\rho>0$ be an arbitrarily fixed number. 
		Then, up to a subsequence, $u_n\weak u$ in $W^{1,p}_G(M)$. Since $G$ is a subgroup of $\mathrm{Isom}_g(M)$,  for every $\xi_1,\xi_2\in G$, by a change of variables, one has 
		$$\|u_n-u\|_{W_g^{1,p}(B_g(\xi_1y,\rho))}=\|u_n-u\|_{W_g^{1,p}(B_g(\xi_2y,\rho))}.$$ 
		Therefore, on account of the definition of $m(y,\rho)$ (see \eqref{myrho}), we have that 
		$$\|u_n-u\|_{W_g^{1,p}(B_g(y,\rho))}\leq \frac{\|u_n-u\|_{W_g^{1,p}(M)}}{m(y,\rho)}.$$ 
		By using Lemma \ref{RK-lemma} and the latter inequality, we obtain 
		$$\|u_n-u\|_{C^0(\overline{B_g(y,\rho)})}\leq \frac{S(y,\rho)}{m(y,\rho)}\|u_n-u\|_{W_g^{1,p}(M)} \leq 
		\frac{S(y,\rho)}{m(y,\rho)} \left( \sup_n \|u_n\|_{W_g^{1,p}(M)}+\|u\|_{W_g^{1,p}(M)} \right).$$ 
		According to {(ii)} and relation \eqref{M-elso} we have that $$\ds 	\lim_{d_g(x_0, y) \to \infty}\frac{S(y,\rho)}{m(y,\rho)}=0,$$ thus for every $\varepsilon>0$ there exists $R_\varepsilon>0$ such that \begin{equation}\label{NagyGombre}\sup_{d_g(x_0,y)\geq R_\varepsilon}\|u_n-u\|_{C^0\left(\overline{B_g(y,\rho)}\right)}\leq \frac{\varepsilon}{2}\ \mbox{ for every }n\in \mathbb{N}.\end{equation}	
		On the other hand, $u_n\weak u$ in $W_G^{1,p}(M)$, thus by the Rellich--Kondrachov-type result (see Lemma \ref{RK-lemma}) it follows that $u_n\to u$ in $C^0\left(\overline{B(y,R_\varepsilon)}\right)$, hence there exists $n_\varepsilon\in \mathbb{N}$ such that \begin{equation}\label{KicsiGombre}\|u_n-u\|_{C^0\left(\overline{B(y,R_\varepsilon)}\right)}<\varepsilon\ \mbox{ for all } \ n\ge n_\varepsilon.
		\end{equation}
		Inequalities \eqref{NagyGombre} and \eqref{KicsiGombre} yield that $u_n\to u$ in $L^\infty(M)$, which concludes the proof. \hfill $\square$

	\begin{remark}\label{remark-2} \rm 
	(a) The quantity $m(y,\rho)$ can be easily estimated on non-positively curved space forms. Indeed, for instance, if $d=2$ and $G=O(2)$, $x_0=0$, then for $\rho>0$ enough small, one has  $m(y,\rho)\sim \frac{\pi|y|}{\rho}$ as $|y|\to \infty$ in the Euclidean case $\mathbb R^2$, and $m(y,\rho)\sim \frac{\pi}{\rho} \frac{|y|}{1-|y|^2}$ as $|y|\to 1$ in the Poincar\'e ball model $\mathbb H^2_{-1}=\{y\in \mathbb R^2:|y|<1\}$ (with constant sectional curvature $-1$).

	(b) 	Relation (\ref{m-tilde-m}) can be viewed as a comparison of the maximal number of mutually disjoint geodesic balls with radius $\rho$  on $(M,g)$ and the Euclidean space, respectively. In fact, $\tilde m(t,\rho) $ is related to the particular inner product given by $g_{x_0}$, which is equivalent to the usual Euclidean metric. This comparison result can be efficiently applied for every Hadamard manifold. 
		In particular, in the usual Euclidean space $\mathbb R^d$, a simple covering argument shows that  $$\tilde m(t,\rho)=\omega\left({V^{-1}_{\rm cap}(2\rho/t)}\right)\ {\rm as}\ \ t\to \infty,\footnote{$f(t)=\omega(g(t))$ as $t\to \infty$ if there exist
		$c,\delta>0$ such that $|f(t)|\geq c|g(t)|$ for every $t>
		\delta$.}$$ 
		where  $V_{\rm cap}(r)$ denotes the area of the spherical cap of radius $r>0$ on the unit  $(d-1)$-dimensional sphere. For instance, when $d=3$, we have $\tilde m(t,\rho)=\omega\left(\sin^{-2}(\rho/t)\right)\ {\rm as}\ \ t\to \infty$.
	
	\end{remark}

\section{Proof of Theorem \ref{mainThm:Riemann2}}\label{4-sect}

\underline{$({\rm i})\Rightarrow ({\rm ii})$} Let us assume by contradiction that $(\mathbf{EC})_G$ fails, i.e. there exist $K\in \mathbb N$ and a sequence $\{x_n\}_n \subset M$ such that $$m(x_n,\rho)\leq K\mbox{ for every }n\in \mathbb N\mbox{ and }d_g(x_0,x_n)\to \infty\mbox{ as }n\to \infty.$$ 
We are going to prove that $x_n\in \mathscr{O}_{4(K+1)\rho}$ for every $n\in \mathbb N,$ which will imply in particular that $\mathscr{O}_{4(K+1)\rho}$ is unbounded, contrary to our assumption. We recall that $\mathscr{O}_{t}=\{x\in M:\:\mathrm{diam}\mathcal{O}_G^x\leq t\}$, $t>0$. 

In order to prove the claim, it suffices to show that $\mathrm{diam}\mathcal{O}_G^{x_n}\leq 4(K+1)\rho$ for every $n\in \mathbb N$. To do this,  let $n\in \mathbb N$ be fixed and $k_n:=m(x_n,\rho)\leq K$. By the definition of $m(x_n,\rho)$, there exist $\xi_i:=\xi_i^n\in G$, $i\in \{1,...,k_n\},$ such that $B_{g}(\xi_{i}x_n,\rho)\cap B_{g}(\xi_{j}x_n,\rho) = \emptyset,\forall\,i\neq j$, $i,j\in \{1,...,k_n\},$ and the number $k_n\in \mathbb N$ is maximal with this property. 

On one hand, if we pick an arbitrary element $\xi\in G$, it follows that there exists $i\in \{1,...,k_n\}$ such that $d_g(\xi x_n, \xi_ix_n)<2\rho$. If this is not the case, i.e., $d_g(\xi x_n, \xi_ix_n)\geq 2\rho$ for every $i\in \{1,...,k_n\}$, it follows that $B_{g}(\xi x_n,\rho)\cap B_{g}(\xi_{i}x_n,\rho)=\emptyset,\forall i \in \{1,...,k_n\},$ i.e., one can find one more element $\xi_{k_n+1}\in G$ with the disjointness property, i.e., 
$B_{g}(\xi_{i}x_n,\rho)\cap B_{g}(\xi_{j}x_n,\rho)=\emptyset,\forall\,i\neq j$, $i,j\in \{1,...,k_n+1\},$
which contradicts the maximality of $k_n=m(x_n,\rho)$. Accordingly, 
$$\mathrm{diam}\mathcal{O}_G^{x_n}\leq 4\rho+ \mathrm{diam}\{\xi_{i}x_n:i\in \{1,...,k_n\}\}.$$

We claim that $\{\xi_{i}x_n:i\in \{1,...,k_n\}\}\subset B_g(\xi_1 x_n,2k_n\rho)$; clearly, we may put any element $\xi_{i}\in G$, $ i\in \{1,...,k_n\}$ instead of $\xi_1\in G$ in the right hand side of the above inclusion. We observe that for $k_n=1$ the claim trivially holds. Thus, let $k_n\geq 2$.   Assume the contrary, i.e., there exists $i_0\in \{2,...,k_n\}$ such that $\xi_{i_0}x_n\notin B_g(\xi_1 x_n,2k_n\rho)$, that is $$d_g(\xi_{i_0}x_n,\xi_1 x_n)\geq 2k_n\rho.$$ We now fix a geodesic segment $\tilde \gamma:[0,1]\mapsto \mathcal{O}_G^{x_n}$ joining the points $\xi_1 x_n\in \mathcal{O}_G^{x_n}$ and $\xi_{i_0}x_n\in \mathcal{O}_G^{x_n}$; this can be done due to the fact that $\mathcal{O}_G^{x_n}$ is a complete connected submanifold of $(M,g)$ (as a closed submanifold of the the complete Riemannian manifold $(M,g)$), see do Carmo \cite[ Corollary 2.10, p. 149]{doCarmo}). Since $d_g(\tilde \gamma(0),\tilde \gamma(1))=d_g(\xi_1 x_n,\xi_{i_0}x_n)\geq 2k_n\rho$, by a continuity reason, we may fix $0<t_1<...<t_{k_n-1}<1$ such that $$d_g(\xi_1x_n,\tilde \gamma(t_j))=2j\rho \mbox{ for every }j\in \{1,...,k_n-1\}.$$ This particular choice clearly shows that $B_g(\tilde \gamma(t_j),\rho)$ are situated in some concentric annuli with the same width; more precisely, 
$$B_g(\tilde \gamma(t_j),\rho)\subset B_g(\xi_1 x_n,(2j+1)\rho)\setminus B_g(\xi_1 x_n,(2j-1)\rho),\ \ j\in \{1,...,k_n-1\}.$$
Beside of the latter property, by $d_g(\xi_{i_0}x_n,\xi_1 x_n)\geq 2k_n\rho$ we also have that 
$$ B_g(\tilde \gamma(1),\rho)\cap  B_g(\xi_1 x_n,(2k_n-1)\rho)=\emptyset.$$ 
Combining all these constructions, it follows that the balls $$B_g(\tilde \gamma(0),\rho)=B_g(\xi_1 x_n,\rho), B_g(\tilde \gamma(t_1),\rho) ...,B_g(\tilde \gamma(t_{k_n-1}),\rho) \mbox{ and }B_g(\tilde \gamma(1),\rho)=B_g(\xi_{i_0}x_n,\rho)$$ are mutually disjoint  sets, whose centers belong to ${\rm Im}\tilde \gamma\subset \mathcal{O}_G^{x_n}$.  Since the number of these balls is $k_n+1$, this contradicts again the maximality of  $k_n=m(x_n,\rho)$. 

Accordingly, $$\mathrm{diam}\mathcal{O}_G^{x_n}\leq 4\rho+ 4k_n\rho\leq 4(K+1)\rho,$$
which concludes the proof. 

\underline{$({\rm ii})\Rightarrow ({\rm iii})$}   We shall focus first on the Morrey-case (\textbf{M}), i.e., we assume that $p>d$ and $q=\infty;$ then we discuss the cases (\textbf{S}) and (\textbf{MT}).

Similarly to \eqref{M-elso}, we are going to prove that for every fixed $\rho>0$ one has
\begin{equation}\label{M-masodik}
{\displaystyle \inf_{y\in M}{\displaystyle S(y,\rho)^{-1}>0}},
\end{equation}
where $S(y,\rho)$ is the embedding constant in  $W^{1,p}_g(B_g(y,\rho))\embd C^0(\overline{B_g(y,\rho)})$, see 	Lemma \ref{RK-lemma}.

We have that for any $\varepsilon>0$ there exists $r_H>0$ depending only on $\varepsilon, d, K$ and $i_{0}$, which satisfies the following property: for any $y\in M$ there exists a harmonic coordinate chart $\varphi:B_g(y,r_H)\to \mathbb{R}^d$, such that $\varphi(y)=0$, and the components $(g_{jl})$ of $g$ in this chart satisfy 
\begin{equation}\label{local_chart1}
\displaystyle \frac{1}{1+\varepsilon}\delta_{jl}\le g_{jl}\le (1+\varepsilon)\delta_{jl}
\end{equation} 
as bilinear forms. Fix $\rho<r_H$, then it is obvious that 
\begin{equation}\label{gombokfontos1}
B_e\left(0,\frac{\rho}{\sqrt{1+\varepsilon}}\right) \ \subseteq \ \Omega_y\coloneqq\varphi\left(B_g(y,\rho)\right) \ \subseteq \  B_e(0,\sqrt{1+\varepsilon}\rho) \ \subset \ \mathbb{R}^d.
\end{equation}
On the other hand, combining  \eqref{egyik1} with \eqref{gombokfontos1}, we have that 
\begin{align*}
S(y,\rho)^{-1} &= \inf_{u\in W_{g}^{1,p}(B_{g}(y,\rho))}\frac{{\displaystyle \left(\int_{B_{g}(y,\rho)}(|\nabla_{g}u|^{p}+|u|^{p})\mathrm{d}v_{g}\right)^{\frac{1}{p}}}}{{\displaystyle \sup_{x\in\overline{B_{g}(y,\rho)}}|u(x)|}}   \\
&\geq (1+\varepsilon)^{-\frac{d+p}{2p}}  \inf_{u\in W_{g}^{1,p}(B_{g}(y,\rho))}\frac{{\displaystyle \left(\int_{\Omega_y}(|\nabla(u \circ \varphi^{-1})|^{p}+|u \circ \varphi^{-1}|^{p}) \mathrm{d}x\right)^{\frac{1}{p}}}}{{\displaystyle \sup_{x\in\overline{\Omega_y}}|u \circ \varphi^{-1}(x)|}} \\
&\geq (1+\varepsilon)^{-\frac{d+p}{2p}} \inf_{f\in W^{1,p}(\Omega_y)}\frac{\|f\|_{W^{1,p}(\Omega_y)}}{\displaystyle \|f\|_{C^0(\overline{\Omega_y})}}.\end{align*} 
Let $f^*: \Omega_y^* \to [0, \infty)$ be the symmetric decreasing rearrangement of the function $f$ (see Lieb and Loss \cite[Section 3.3]{LiebLoss}), thus ${\rm Vol}_e(\Omega_y)={\rm Vol}_e(\Omega_y^*)$ and
$$\inf_{f\in W^{1,p}(\Omega_y)}\frac{\|f\|_{W^{1,p}(\Omega_y)}}{\displaystyle \|f\|_{C^0(\overline{\Omega_y})}}\geq \inf_{f^*\in W^{1,p}(\Omega_y^*)}\frac{\|f^*\|_{W^{1,p}(\Omega_y^*)}}{\displaystyle \|f^*\|_{C^0(\overline{\Omega_y^*})}}=\inf_{f^*\in W^{1,p}(\Omega_y^*)}\frac{\|f^*\|_{W^{1,p}(\Omega_y^*)}}{\displaystyle f^*(0)}.$$ 
Since $B_e\left(0,\frac{\rho}{\sqrt{1+\varepsilon}}\right)\subseteq\Omega_y^*\subseteq B_e(0,\sqrt{1+\varepsilon}\rho)\subset \mathbb{R}^d$, we have that 
$$W^{1,p}\left(B_e\left(0,\frac{\rho}{\sqrt{1+\varepsilon}}\right)\right)\supseteq W^{1,p}(\Omega_y^*)\supseteq W^{1,p}(B_e(0,\sqrt{1+\varepsilon}\rho)).$$ 
Hence
$$\inf_{f^*\in W^{1,p}(\Omega_y^*)}\frac{\|f^*\|_{W^{1,p}(\Omega_y^*)}}{\displaystyle f^*(0)}\geq \inf_{f^*\in W^{1,p}\left(B_e\left(0,\frac{\rho}{\sqrt{1+\varepsilon}}\right)\right)}\frac{\|f^*\|_{W^{1,p}\left(B_e\left(0,\frac{\rho}{\sqrt{1+\varepsilon}}\right)\right)}}{\displaystyle \|f^*\|_{C^0\left(\overline{B_e\left(0,\frac{\rho}{\sqrt{1+\varepsilon}}\right)}\right)}}>0,$$	
meaning that ${\displaystyle \inf_{y\in M}S(y,\rho)^{-1}>0}$, which concludes the proof of \eqref{M-masodik}.

 Now, let  $\{u_n\}_n \subset W^{1,p}_G(M)$ be a bounded sequence and $\rho>0$ be an arbitrarily fixed number.  
Then, up to a subsequence, $u_n\weak u$ in $W^{1,p}_G(M)$. 
By using Lemma \ref{RK-lemma}, we obtain 
$$\|u_n-u\|_{C^0(\overline{B_g(y,\rho)})}\leq \frac{S(y,\rho)}{m(y,\rho)}\|u_n-u\|_{W_g^{1,p}(M)} \leq 
\frac{S(y,\rho)}{m(y,\rho)} \left( \sup_n \|u_n\|_{W_g^{1,p}(M)}+\|u\|_{W_g^{1,p}(M)} \right).$$ 
Due to the validity of $(\mathbf{EC})_G$ and relation \eqref{M-masodik} we have that $$\ds 	\lim_{d_g(x_0, y) \to \infty}\frac{S(y,\rho)}{m(y,\rho)}=0,$$ thus for every $\varepsilon>0$ there exists $R_\varepsilon>0$ such that \begin{equation}\label{NagyGombre1}\sup_{d_g(x_0,y)\geq R_\varepsilon}\|u_n-u\|_{C^0\left(\overline{B_g(y,\rho)}\right)}\leq \frac{\varepsilon}{2}\ \mbox{ for every }n\in \mathbb{N}.\end{equation}	
Since $u_n\weak u$ in $W_G^{1,p}(M)$,  by the Rellich--Kondrachov-type result (see Lemma \ref{RK-lemma}) it follows that $u_n\to u$ in $C^0\left(\overline{B(y,R_\varepsilon)}\right)$, hence there exists $n_\varepsilon\in \mathbb{N}$ such that \begin{equation}\label{KicsiGombre1}\|u_n-u\|_{C^0\left(\overline{B(y,R_\varepsilon)}\right)}<\varepsilon\ \mbox{ for all } \ n\ge n_\varepsilon.
\end{equation}
Inequalities \eqref{NagyGombre1} and \eqref{KicsiGombre1} yield that $u_n\to u$ in $L^\infty(M)$,  ending the proof in the admissible case (\textbf{M}).

Let us fix an arbitrary $d$-admissible pair $(p,q)$ from (\textbf{S}) or (\textbf{MT}). A suitable modification of the above argument, based on Lemma \ref{PolyaSzego}/(i), implies that
$$S(y,\rho)^{-1} := \inf_{u\in W_{g}^{1,p}(B_{g}(y,\rho))}\frac{{\displaystyle \left(\int_{B_{g}(y,\rho)}(|\nabla_{g}u|^{p}+|u|^{p})\mathrm{d}v_{g}\right)^{\frac{1}{p}}}}{{\displaystyle \left(\int_{B_{g}(y,\rho)}|u|^{q}\mathrm{d}v_{g}\right)^{\frac{1}{q}}}}>0.$$
The latter inequality together with the validity of $(\mathbf{EC})_G$ implies that
$$\ds 	\lim_{d_g(x_0, y) \to \infty}\frac{S(y,\rho)}{m(y,\rho)}=0.$$
The rest is analogous as before, by using the Rellich--Kondrachov compactness result  $W^{1,p}(B_g(y,R))\hookrightarrow L^q(B_g(y,R))$ for any $R>0$ fixed. 

\underline{$({\rm iii})\Rightarrow ({\rm iv})$} Trivial. 

\underline{$({\rm iv})\Rightarrow ({\rm i})$} We follow the argument presented in Skrzypczak and Tintarev \cite[Theorem 4.3]{Skrzypczak2} (see also Tintarev \cite[Theorem 7.10.12]{TintarevKonyv}); in fact, the admissible case (\textbf{S}) is exactly the one proved in Tintarev \cite{TintarevKonyv}. Since the case (\textbf{MT}) can be similarly discussed as (\textbf{S}), we restrict our proof to the remaining admissible case (\textbf{M}). 

 Suppose that $G$ is not coercive, thus there exists $R>0$ and a discrete sequence of $x_n\in M$, such that $\mathcal{O}_G^{x_n}\subset B_g(x_n,R)$ and $d_g(x_0,x_n)\to \infty$ as $n\to \infty$.  Let $r \in  (0, \mathrm{inj}_{(M,g)})$ and let us replace $\{x_n\}_n$ with a renumbered subsequence such that distance between any two terms in the sequence will be greater than $2(R + r)$. We define a sequence of functions $\{f_n\}_n$ by 
 $$f_n(x)=\int_{G} (r-d_g(\xi x,x_n))_+\,\mathrm{d}\xi,$$ 
where the Haar measure of $G$ is normalized to the value $1$, and $u_+=\max\{0,u\}$. It is easy to see that $f_n\in W^{1,p}_G(M)$ for every $n\in \mathbb N$ and any fixed $p\in (1,\infty)$; indeed, since 
the support of $f_n$ is a subset of  $\overline{B_g(\xi^{-1}x_n,r)}$, by an elementary computation (with \eqref{dist-gradient} in hand) and the volume-estimate \eqref{volume-comp-altalanos-1}, it follows that 
$$\|f_n\|_{W_g^{1,p}(M)}\leq C(p,r,d),$$ where $C(p,r,d)>0$ does not independent on $n$.
On one hand, since the supports of the functions $f_n$ are disjoint sets, we have that $$\|f_l-f_n\|_{L^\infty(M)}=\|f_l\|_{L^\infty(M)}+\|f_n\|_{L^\infty(M)}\geq 2\inf_n \|f_n\|_{L^\infty(M)},\ \ l\neq n.$$
On the other hand, 
\begin{align*}
	\mathrm{Vol}_g(B_g(x_n,R+r)) \|f_n\|_{L^\infty(M)}  \geq \int_M f_n(x)\,\mathrm{d}v_g =&\int_M \int_{G} (r-d_g(\xi x,x_n))_+\,\mathrm{d}\xi \,\mathrm{d}v_g(x)\\
	=& \int_G \int_{M} (r-d_g(\xi x,x_n))_+\,\mathrm{d}v_g(x)\,\mathrm{d}\xi \\
	\overset{x \coloneqq \xi^{-1}y}{\ \ \ \ \ \ =}&\int_G \int_{M} (r-d_g(y,x_n))_+\,\mathrm{d}v_g(y)\,\mathrm{d}\xi \\
	=&\int_{M} (r-d_g(y,x_n))_+\,\mathrm{d}v_g(y)\\
	\geq&~ \frac{r}{2}\mathrm{Vol}_g\left(B_g\left(x_n,\frac{r}{2}\right)\right). 
\end{align*}
Since $(M,g)$ is a Riemannian manifold with bounded geometry, then ${\rm Vol}_g$ is doubling on $(M,g)$, thus $$\|f_n\|_{L^\infty(M)}\geq \tilde C(r,R,d),$$ where $\tilde C(r,R,d)>0$ does not independent on $n$. Thus,  $\{f_n\}_n$ is not a Cauchy sequence in $L^\infty(M)$,  a contradiction. \hfill $\square$\\

%
%
%

	The following theorem is related to the results obtained in Hebey and Vaugon \cite{HebeyVaugon}:
	\begin{theorem}\label{key_RiemannLemma2}
		Let $(M,g)$ be a $d$-dimensional complete non-compact Riemannian manifold with Ricci curvature bounded from below having positive injectivity radius, and let $G$ be a compact connected subgroup of $\mathrm{Isom}_{g}(M)$ such that $\mathrm{Fix}(G)=\{x_0\}$ for some $x_0\in M$ and  $\rho>0$ be small enough. Assume that there exists $\kappa=\kappa(G,d) > 0$  such that for every $y\in M$ with $d_g(x_0,y) \geq 1$, one has
		\begin{equation*}\mathcal H^{l}(\mathcal{O}_G^y) \geq \kappa \cdot d_g(x_0,y),\eqno{(\mathbf{H})}
		\end{equation*}
		where $l=l(y)=\dim \mathcal{O}_G^y \geq 1$. Then the embedding $W_{G}^{1,p}(M)\hookrightarrow L^{\infty}(M)$	is compact for every $p>d$.  
	\end{theorem}

	\begin{proof}
		Let $y\in M$ be arbitrarily fixed such that $d_g(x_0,y) \geq 1$, and consider the elements $\xi_i\in G$, $i=1,\dots,m(y,\rho)$ which appear in the definition of $m(y,\rho)$ in \eqref{myrho}. Let also $l=l(y)=\dim \mathcal{O}_G^y.$
		Notice that by the connectedness of $G$, we have $l \geq 1$. 
		We claim that 
		\begin{equation}\label{area-0}
			\mathcal H^{l}(\mathcal{O}_G^y)\leq m(y,\rho)\sup_{i}\mathcal H^{l}(B_g(\xi_i y,k\rho)\cap \mathcal{O}_G^y),
		\end{equation}
		for every $k>2$ (independent of $y$). 
		To see this, it is sufficient to prove that 
		$$\mathcal{O}_G^y\subseteq \bigcup_{i}\left(B_g(\xi_i y,k\rho)\cap \mathcal{O}_G^y\right).$$
		
		Let us fix $x\in \mathcal{O}_G^y$ arbitrarily. 	First, if 
		$\displaystyle x\in \underset{i}{\bigcup}\left(B_g(\xi_i y, \rho)\cap \mathcal{O}_G^y\right)$, we have nothing to prove. If  $\displaystyle x\notin  \underset{i}{\bigcup}\left(B_g(\xi_i y, \rho)\cap \mathcal{O}_G^y\right),$ then there exists $i_0\in\{1,\dots,m(y,\rho)\}$ such that 
		$$d_g\left(x,\partial \left(B_g(\xi_{i_0} y,\rho) \cap \mathcal{O}_G^y \right)\right)<\rho.$$
		Indeed, if the contrary holds, then  $B_g(x,\rho)\cap B_g(\xi_i y,\rho)=\emptyset$, $i=1,\dots,m(y,\rho)$, thus $B_g(x,\rho)$ is a new ball in the definition of $m(y,\rho)$, contradicting the maximality of $m(y,\rho)$. 
		Therefore, $\displaystyle d_g(x,\xi_{i_0} y)<2\rho$, which means that $x\in B_g(\xi_{i_0}y,k\rho)\cap \mathcal{O}_G^y$ for every $k>2$, which proves (\ref{area-0}). 
		
		We also notice that since  Fix$(G)=\{x_0\}$, one has that $\mathcal{O}_G^y\subset \partial B_g(x_0,d_g(x_0,y))$. Indeed, if $x=\xi y\in \mathcal{O}_G^y$ then $d_g(x_0,x)=d_g(x_0,\xi y)=d_g(\xi x_0,\xi y)=d_g(x_0, y).$ Thus $\mathcal{O}_G^y$ is an $l$-dimensional submanifold of $\partial B(x_0,d_g(x_0,y))$, $l \leq d-1$. 
		Therefore, a slight modification of Gallot, Hulin, and Lafontaine \cite[Theorem 3.98]{GallotHulinLafontaine} gives that for every $i=1,\dots,m(y,\rho)$,  
		$$\mathcal H^{l}(B_g(\xi_i y,k\rho)\cap \mathcal{O}_G^y)\leq k^{l}\omega_{l} \rho^{l}(1+o(\rho))\ \ {\rm as}\ \ \rho\to 0,$$ 
		whenever $k>2$ is kept small (e.g. $k=3$). To see this, we explore  that $\exp_{\xi_i y}:T_{{\xi_i y}}M\to M$ is a local diffeomorphism at $0\in T_{{\xi_i y}}M$ with $d(\exp_{\xi_i y})_0=id$, while for small $\rho>0$ one has $\exp_{\xi_i y}^{-1}(B_g(\xi_i y,k\rho)\cap \mathcal{O}_G^y)=B_e(0,k\rho)\cap \exp_{\xi_i y}^{-1}(\mathcal{O}_G^y)$, and $0<\mathcal H^l(\exp_{\xi_i y}^{-1}(\mathcal{O}_G^y))<\infty$.
		
		Now, if we fix $\rho\in (0,1)$ from the usual range (see Gallot, Hulin, and Lafontaine \cite{GallotHulinLafontaine}), it follows by (\ref{area-0})  that 
		$$\mathcal H^{l}(\mathcal{O}_G^y)\leq m(y,\rho)k^{l+1}\omega_{l} \rho^{l}.$$
		Hypothesis ($\mathbf{H}$) and the latter estimate imply that 
		$\kappa \cdot d_g(x_0,y)\leq m(y,\rho)k^{l+1}\omega_{l} \rho^{l}.$
		By using this inequality, one can obtain an $l=l(y)$-independent estimate, namely
		$$\kappa \cdot d_g(x_0,y)\leq m(y,\rho)k^{d}\omega_{d-1} \rho.$$
		Letting $d_g(x_0,y)\to \infty$ 
		immediately implies that $m(y,\rho)\to \infty$. The rest of the proof is similar to the last part of the proof of Theorem \ref{mainThm:Riemann2}.
	\end{proof}
	
	In the sequel, we provide two examples where hypothesis \hyperref[H2felt]{$(\mathbf{H})$} holds. 
	
	\begin{example}\rm 
		Let ${\rm Sym}(d,\mathbb{R})$ be the set of symmetric $d\times d$
		matrices with real values, ${\rm P}(d,\mathbb{R})\subset{\rm Sym}(d,\mathbb{R})$
		be the cone of symmetric positive definite matrices, and ${\rm P}(d,\mathbb{R})_{1}$
		be the subspace of matrices in ${\rm P}(d,\mathbb{R})$ with determinant
		one. The set ${\rm P}(d,\mathbb{R})$ is endowed with the scalar product
		\[
		\langle U,V\rangle_{X}={\rm Tr}(X^{-1}VX^{-1}U)\ \ {\rm for\ all}\ \ X\in{\rm P}(d,\mathbb{R}),\ U,V\in T_{X}({\rm P}(d,\mathbb{R}))\simeq{\rm Sym}(d,\mathbb{R}),
		\]
		where ${\rm Tr}(Y)$ denotes the trace of $Y\in{\rm Sym}(d,\mathbb{R})$.
		One can prove that $({\rm P}(d,\mathbb{R})_{1},\langle\cdot,\cdot\rangle)$
		is a  Riemannian manifold (with non-constant sectional
		curvature). On the other hand, since the scalar curvature of the Riemannian
		manifold $({\rm P}(d,\mathbb{R})_{1},\langle\cdot,\cdot\rangle)$
		is constant, $S=-\frac{1}{8}d(d-1)(d+2)$, see Andai \cite{Andai} and Moakher and Z\'era\"i \cite{MM11}, it follows that its Ricci curvature is bounded from below.
		
		The special linear group $SL(d)$ leaves ${\rm P}(d,\mathbb{R})_{1}$
		invariant and acts transitively on it. Moreover, for every $\sigma\in{SL}(d)$,
		the map $[\sigma]:{\rm P}(d,\mathbb{R})_{1}\to{\rm P}(d,\mathbb{R})_{1}$
		defined by $[\sigma](X)=\sigma X\sigma^{t}$, is an isometry, where
		$\sigma^{t}$ denotes the transpose of $\sigma.$ If $G={SO}(d)$,
		we can prove that ${\rm Fix}_{{\rm P}(d,\mathbb{R})_{1}}(G)=\{I_{d}\}$,
		where $I_{d}$ is the identity matrix; for more details, see Krist\'aly
		\cite{KristalyJFA}. On the other hand, the metric function on ${\rm P}(d,\mathbb{R})$
		is given by
		$
		d_{P}(X,Y)=\sqrt{\mathrm{Tr}\left(\ln^{2}\left(X^{-\frac{1}{2}}YX^{-\frac{1}{2}}\right)\right)},  
		$ see Krist\'aly \cite{KristalyJMPA}.
		
		For simplicity, fix $d=2$, and consider the following positive-definite symmetric
		matrix 
		\[
		X=\left(\begin{array}{cc}
			a & b\\
			b & c
		\end{array}\right),\ a,c>0,\ \mbox{ and }ac-b^{2}=1.
		\]
		Thus
		\[
		\mathcal{O}_{G}^{X}=\left\{ X\xi:\ \xi=\left(\begin{array}{cc}
			\cos\theta & \sin\theta\\
			-\sin\theta & \cos\theta
		\end{array}\right),\theta\in [0,2\pi]\right\} .
		\]
		One can see that 
		\[
		\mathcal{H}^{1}\left(\mathcal{O}_{G}^{X}\right)=2\pi\|X\|_{F}=
		2\pi\sqrt{a^{2}+2b^{2}+c^{2}}\ \ {\rm and}\ \ 	d_{P}(I_{2},X)=\sqrt{\mathrm{Tr}\left(\ln^{2}\left(X\right)\right)}=\sqrt{\ln^{2}(\lambda_{1})+\ln^{2}(\lambda_{2})},\]
		where $\lambda_{1},\lambda_{2}$ are the positive eigenvalues of the
		matrix $X$. Since $
		\sqrt{a^{2}+2b^{2}+c^{2}}=\sqrt{\lambda_{1}^{2}+\lambda_{2}^{2}},
		$ 
		by using a Bernoulli-type inequality, it turns out that 
		$
		\mathcal{H}^{1}\left(\mathcal{O}_{G}^{X}\right)\geq\kappa d_{P}(I_{2},X),
		$
		with  $\kappa:=\pi$, 
		which proves the validity of ($\mathbf{H}$).
	\end{example}

	\begin{example}\rm 
		Let $G=O(d_1)\times\dots\times O(d_k)$ with $d_i\geq 2$, $i=1,\dots,k,$ and $d_1+\dots+d_k=d$. Let $y=(y_1,\dots,y_k)\in \mathbb R^{d_1}\times\dots\times \mathbb R^{d_k}$. It is clear that $\mathcal{O}_G^y=S_{|y_1|}^{d_1-1}\times\dots\times S_{|y_k|}^{d_k-1},$
		where $S_r^{\alpha-1}$ denotes the sphere with radius $r>0$ in $\mathbb R^\alpha$. Let $I(y)=\{i\in \{1,\dots,k\}:|y_i|\neq 0\}$. Then $\displaystyle l=l(y)=\sum_{i\in I(y)}(d_i-1)$ and  $$\mathcal H^{l}(\mathcal{O}_G^y)=\sum_{i\in I(y)}\mathcal H^{d_i-1}(S_{1}^{d_i-1})|y_i|^{d_i-1}\geq 2\pi \sum_{i\in I(y)}|y_i|^{d_i-1}=2\pi \sum_{i=1}^k|y_i|^{d_i-1}.$$
		Now, let $|y_1|+\dots+|y_k|=c\geq 1$. By the scaling $y_i \coloneqq cz_i$, one has $|z_1|+\dots+|z_k|= 1$
		and $$\sum_{i=1}^k|y_i|^{d_i-1}\geq c \sum_{i=1}^k|z_i|^{d_i-1}.$$
		Note that the continuous function $\displaystyle (z_1,\dots,z_k)\mapsto \sum_{i=1}^k|z_i|^{d_i-1}$ attains its minimum on the simplex $|z_1|+\dots+|z_k|= 1$, and this minimum is strictly positive, say $m_G>0$ (otherwise, if $m_G=0$, we would have all variables equal to zero, which is a contradiction). Summing up, it follows that 
		$$\mathcal H^{l}(\mathcal{O}_G^y)\geq 2\pi c m_G=2\pi m_G(|y_1|+\dots+|y_k|)\geq2\pi m_G |y|,$$
		thus $G$ satisfies the assumption in $(\mathbf{H})$. 
	\end{example}

	\section{Sobolev-type embeddings on Randers spaces} \label{sec:Finsler}

	\subsection{Elements from Finsler geometry} 
	
	Let $M$ be a smooth, $d$-dimensional manifold and $TM=\bigcup_{x \in 
		M}T_{x} M $ its tangent bundle. Throughout this subsection,
	the function $F: TM\to [0,\infty)$ is given by the {\it Randers metric}
	\begin{equation}\label{Randers-metrika}
		F(x,y)=\sqrt{g_x(y,y)}+\beta_x(y),\ (x,y)\in TM,
	\end{equation}
	where $g$ is a Riemannian metric on $M$, $\beta_x$ is a $1$-form on
	$M$, and we assume that
	$$\|\beta\|_g(x)=\sqrt{g_x^*(\beta_x,\beta_x)}<1,\ \forall x\in M.$$
	Here, the co-metric $g^*_x$ can be identified by the inverse of the
	symmetric, positive definite matrix $g_x$. The pair $(M,F)$ is called a
	{\it Randers space}, which is a typical Finsler manifold, i.e. the
	following properties hold:
	\begin{itemize}
		
		\item[(a)] $F\in C^{\infty}(TM\setminus\{ 0 \});$
		
		\item[(b)] $F(x,ty)=tF(x,y)$ for all $t\geq 0$ and $(x,y)\in TM;$
		
		\item[(c)] $h_{(x,y)}=[h_{ij}(x,y)]:=\left[\frac12F^{2}%
		(x,y)\right]_{y^{i}y^{j}}$ is positive definite for all $(x,y)\in
		TM\setminus\{ 0 \},$
	\end{itemize}
	see Bao, Chern, and Shen \cite{BaoChernShen}. Clearly, the Randers
	metric $F$ is symmetric, i.e.
	$F(x,-y)=F(x,y)$ for every $(x,y)\in TM,$ if and only if $\beta=0$ (which means that $(M,F)=(M,g)$ is the original Riemannian manifold).

	Let $\sigma: [0,r]\to M$ be a piecewise $C^{\infty}$ curve. The
	value $ L_F(\sigma)= \displaystyle\int_{0}^{r} F(\sigma(t),
	\dot\sigma(t))\,{\text d}t $ denotes the \textit{integral length} of
	$\sigma.$  For $x_1,x_2\in M$, denote by $\Lambda(x_1,x_2)$ the set
	of all piecewise $C^{\infty}$ curves $\sigma:[0,r]\to M$ such that
	$\sigma(0)=x_1$ and $\sigma(r)=x_2$. Define the {\it distance
		function} $d_{F}: M\times M \to[0,\infty)$ by
	\begin{equation}\label{quasi-metric}
		d_{F}(x_1,x_2) = \inf_{\sigma\in\Lambda(x_1,x_2)}
		L_F(\sigma).
	\end{equation}
	One clearly has that $d_{F}(x_1,x_2) =0$ if and only if $x_1=x_2,$
	and that $d_F$ verifies the triangle inequality.
	
	The {\it Hausdorff  volume form} ${\text d}V_F$ on the Randers space $(M,F)$ is given by
	\begin{equation}\label{randers-volume-def}
		{\text d}V_F(x)=\left(1-\|\beta\|^2_g(x)\right)^\frac{d+1}{2}{\rm d}v_{g},
	\end{equation}
	where ${\rm d}v_{g}$ denotes the canonical Riemannian volume form induced by $g$ on $M$.
	
	For every $(x,\alpha)\in T^*M$, the  {\it polar transform} (or,
	co-metric) of $F$ from (\ref{Randers-metrika}) is
	\begin{equation}\label{polar-transform}
		F^*(x,\alpha) = \sup_{y\in T_xM\setminus
			\{0\}}\frac{\alpha(y)}{F(x,y)}=\frac{\sqrt{g_x^{*2}(\alpha,\beta)+(1-\|\beta\|_g^2(x))\|\alpha\|_g^2(x)}-g_x^{*}(\alpha,\beta)}{1-\|\beta\|_g^2(x)}.
	\end{equation}
	
	Let $u:M\to \mathbb{R}$ be a differentiable function in the distributional sense. The \textit{gradient} of $u$ is defined by
	\begin{equation}  \label{grad-deriv}
		\boldsymbol{\nabla}_F u(x)=J^*(x,Du(x)),
	\end{equation}
	where $Du(x)\in T_x^*M$ denotes the (distributional) \textit{derivative} of $u$ at $x\in M$ and $J^*$ is the Legendre transform  given by $$J^*(x,y):=\frac{\partial}{\partial y}\left(\frac{1}{2}F^{*2}(x,y)\right).$$  In local coordinates, one has
	\begin{equation}  \label{derivalt-local}
		Du(x)=\sum_{i=1}^n \frac{\partial u}{\partial x^i}(x)\mathrm{d}x^i,
	\end{equation}
	\begin{equation*}
		\boldsymbol{\nabla}_F u(x)=\sum_{i,j=1}^n h_{ij}^*(x,Du(x))\frac{\partial u}{\partial x^i}(x)\frac{\partial}{\partial x^j}.
	\end{equation*}
	
	In general, note that $u\mapsto\boldsymbol{\nabla}_F u $ is not linear. If $x_0\in M$ is
	fixed, then due to Ohta and Sturm \cite{Otha-Sturm}, one has 
	\begin{equation}  \label{tavolsag-derivalt}
		F^*(x,D d_F(x_0,x))=F(x,\boldsymbol{\nabla}_F d_F(x_0,x))=D d_F(x_0,x)(%
		\boldsymbol{\nabla}_F d_F(x_0,x))=1\ \mathrm{for\ a.e.}\ x\in M.
	\end{equation}
	
	Let $X$ be a vector field on $M$. In a local coordinate system $(x^i)$  the \textit{divergence} is defined by div$(X)=\frac{1}{\sigma_F}\frac{\partial}{\partial x^i}(\sigma_F X^i),$ where 
	$$\sigma_F(x)=\frac{\omega_{d}}{\mathrm{Vol}(\{y=(y^i):\ F(x,y^i\frac{\partial }{\partial x^i})<1\})}.$$ 
	

	The Finsler $p$-Laplace operator is defined by $$\boldsymbol{\Delta}_{F,p}u=\mathrm{div}({F^*}^{p-2}(Du) \cdot \boldsymbol{\nabla}_F u),$$
	while the Green theorem reads as: for every $v\in C_0^\infty(M)$,
	\begin{equation}  \label{Green}
		\int_M v\boldsymbol{\Delta}_{F,p} u \,{\mathrm d}V_F(x)=-\int_M
		{F^*}^{p-2}(Du) Dv(\boldsymbol{\nabla}_F u)\,{\text d}V_F(x),
	\end{equation}
	see Ohta and Sturm \cite{Otha-Sturm} and 
	Shen \cite{Shen_Konyv} for $p=2$. Note that in general 
	$\boldsymbol{\Delta}_{F,p} (-u) \neq -\boldsymbol{\Delta}_{F,p} u.$
	When $(M,F)=(M,g)$, the Finsler-Laplace operator is the usual Laplace-Beltrami operator, 
	
	We introduce the Sobolev space
	associated with $\left(M,F\right)$, namely let 
	$$W^{1,p}_F(M)=\left\{u\in W^{1,p}_\mathrm{loc}(M): \int_M {F^*}^p(x,Du(x))\mathrm{d}V_F(x)<+\infty \right\}$$ 
	be the closure of $C^\infty(M)$ with respect to the (asymmetric) norm 
	$$\|u\|_{W^{1,p}_F(M)}=\left(\int_M {F^*}^p(x,Du(x))\,\mathrm{d}V_F(x)+\int_M |u(x)|^p\,\mathrm{d}V_F(x) \right)^{\frac{1}{p}}.$$

	We notice that  the \textit{reversibility constant} associated with $F$ (see (\ref{Randers-metrika})) is given by 
	\begin{equation}  \label{reverzibilis}
		r_{F}=\sup_{x\in M}r_F(x)\ \ \ \mathrm{where}\ \ \ r_F(x)=\sup_{\substack{ y \in T_x M\setminus \{0\}}} \frac{F(x,y)}{F(x,-y)}=\frac{1+\|\beta\|_g(x)}{1-\|\beta\|_g(x)},
	\end{equation}
	see Rademacher \cite{Rademacher} and Zhao and Yuan \cite{Yuan-Zhao}.
	Note that $r_{F}\geq 1$ (possibly, $r_{F}=+\infty$), and $r_{F}= 1$ if
	and only if $(M,F)$ is Riemannian. 
	Analogously, the \textit{uniformity constant} of $F$ is defined by the number
	\begin{equation} \label{uniformity_const}
		l_{F}=\inf_{x\in M}l_F(x)\ \ \ \mathrm{where}\ \ \ l_F(x)= \inf_{y,v,w\in
			T_xM\setminus \{0\}}\frac{h_{(x,v)}(y,y)}{h_{(x,w)}(y,y)}=\left(\frac{1-\|\beta\|_g(x)}{1+\|\beta\|_g(x)}\right)^2,
	\end{equation}
	and measures how far $F$ and $F^*$ are from Riemannian structures, see Egloff \cite{Egloff}. Note that $l_{F}\in [0,1]$, and $l_{F}= 1$ if
	and only if $(M,F)$ is Riemannian, i.e. $\beta = 0$. 
	%

	\subsection{Embedding results on Randers spaces: the influence of reversibility}


	\begin{proof}[Proof of Theorem \ref{Finsler}]
		Let $(M,F)$ be a $d$-dimensional Randers space with 
		$$F(x,y) = \sqrt{g_x(y,y)}+\beta_x(y),\, (x,y)\in TM,$$ 
		where $g$ is a Riemannian metric such that $(M,g)$ is either a Hadamard manifold or a Riemannian manifold with bounded geometry. Let $$\displaystyle a \coloneqq \sup_{x\in M}\|\beta\|_g(x)<1.$$ 
		In this case, the volume form on $(M,F)$ is given by \eqref{randers-volume-def}, one has that
		\begin{equation}\label{randers-volume}
			(1-a^2)^\frac{d+1}{2}{\rm d}v_{g}\leq {\text d}V_F(x)\leq {\rm d}v_{g} .
		\end{equation}
		
		Next, by using the definition of the polar transform of $F$, see \eqref{polar-transform}, we get that 
		\begin{align}F^*(x,\alpha)&\leq \frac{\sqrt{\|\alpha\|_g^2(x)\cdot \|\beta\|_g^2(x)+(1-\|\beta\|_g^2(x))\|\alpha\|_g^2(x)}+\|\alpha\|_g(x)\cdot \|\beta\|_g(x)}{1-\|\beta\|_g^2(x)} \nonumber \\ &= \frac{\|\alpha\|_g(x)(1+\|\beta\|_g(x))}{1-\|\beta\|_g^2(x)}=\frac{\|\alpha\|_g(x)}{1-\|\beta\|_g(x)}\leq \frac{\|\alpha\|_g(x)}{1-a}\label{egyikbecsles}.
		\end{align}
		
		On the other hand,    
		\begin{align}
			F^*(x,\alpha)&=\frac{\|\alpha\|_g^2(x)}{\sqrt{g_x^{*2}(\alpha,\beta)+(1-\|\beta\|_g^2(x))\|\alpha\|_g^2(x)}+g_x^{*}(\alpha,\beta)} \nonumber \\
			&\geq \frac{\|\alpha\|_g^2(x)}{\sqrt{\|\alpha\|_g^2(x)\cdot \|\beta\|_g^2(x)+(1-\|\beta\|_g^2(x))\|\alpha\|_g^2(x)}+\|\alpha\|_g(x)\cdot \|\beta\|_g(x)}   \nonumber \\
			&= \frac{\|\alpha\|_g(x)}{1+\|\beta\|_g(x)} 
			\ \geq \ \frac{\|\alpha\|_g(x)}{1+a}.\label{masikbecsles} 
		\end{align}
		
		Combining \eqref{randers-volume}, \eqref{egyikbecsles} and \eqref{masikbecsles}, one gets that 
		\begin{equation}\label{normaekvivalencia}
			\frac{(1-a^2)^\frac{d+1}{2}}{(1+a)^p}\|u\|^p_{W^{1,p}_g(M)}\leq \|u\|^p_{W^{1,p}_F(M)}\leq \frac{1}{(1-a)^p}\|u\|^p_{W^{1,p}_g(M)}.
		\end{equation}

		Thus, by the continuous  embedding on the Riemannian manifold, we have that $$\|u\|^p_{W^{1,p}_F(M)}\geq \frac{(1-a^2)^\frac{d+1}{2}}{(1+a)^p} C \|u\|_{L^q(M)},$$ where $(p,q)$ is any $d$-admissible pair. 
		
		For the compact embedding, let $G$ be a compact connected subgroup of $\mathrm{Isom}_F(M)$, such that $m_F(y,\rho)\to\infty\ \mbox{ as }\ d_{F}(x_{0},y)\to\infty$ for some $x_0\in M$. According to Deng \cite[Proposition 7.1]{DengKonyv}, $G$ is a closed subgroup of the isometry group of the Riemannian manifold $(M,g)$.  
		
	 On the other hand, since $(1-a)d_g(x_0,y)\leq d_F(x_0,y)\leq (1+a)d_g(x_0,y)$, we have that if $m_F(y,\rho)\to\infty\ \mbox{ as }\ d_{F}(x_{0},y)\to\infty$, then $m\left(y,\frac{\rho}{(1+a)}\right)\to \infty$ as $d_g(x_0,y)\to \infty$.

	Now let $\{u_n\}_n$ be a bounded sequence in $W^{1,p}_{F,G}(M)$. From \eqref{normaekvivalencia}, it follows that $\{u_n\}_n$ is bounded in $W^{1,p}_G(M)$, thus by Theorems \ref{mainThm:Riemann1} \& \ref{mainThm:Riemann2} (condition $(\mathbf{EC})_G$ implies the compact embedding), there exists a subsequence $\{u_{n_k}\}_k$ which converges strongly to a function $u$ in $L^q(M)$, where $(p,q)$ is any $d$-admissible pair. This  concludes the proof.
	\end{proof}

	We emphasize that Theorem \ref{Finsler} is sharp in the following sense: if we consider the $d$-dimensional Finslerian Funk model $(B^d(1), F)$, which is a non-compact Finsler manifold of Randers-type having constant flag curvature $-\frac{1}{4}$, then we can construct a function $u \in W_F^{1,p}(B^d(1))$ such that $\|u\|_{L^q(B^d(1))}=+\infty$, in other words, $W_F^{1,p}(B^d(1))\bcancel{\hookrightarrow} L^q(B^d(1))$, for any $(p,q)$ $d$-admissible pair. As it turns out, $(B^d(1), F)$ is a non-reversible Finsler manifold with $\displaystyle \sup_{x\in M}\|\beta\|_g(x) = 1$, i.e. $r_F = \infty$, see Krist\'aly and Rudas \cite{KristalyRudas_NATMA}. Therefore, the continuous embeddings of Sobolev spaces do not necessarily hold on Randers spaces having infinite reversibility constant. Details are provided in the next example.

	\begin{example}\label{FinslerPelda} \rm 
		Let $d\geq 2$, $(p,q)$ be a $d$-admissible pair, and $B^d(1) = \{x \in \mathbb{R}^d : |x| < 1\}$ be the $d$-dimensional Euclidean open unit ball. 
		Consider the Funk metric $F :B^d(1)\times \mathbb{R}^d \to \mathbb{R}$ defined by
		$$F(x,y)=\frac{\sqrt{|y|^2-\left(|x|^2|y|^2-\langle x,y\rangle^2\right)}}{1-|x|^2}+\frac{\langle x,y\rangle}{1-|x|^2}.$$ 
		The pair $(B^d(1), F)$ is the Finslerian Funk model, see Cheng and Shen \cite[Example 2.1.2]{ChengShenkonyv}, and Shen \cite[Example 1.3.4]{Shen_Konyv}. According to  Shen \cite{Shen_Konyv}, we have that $\displaystyle d_F(0,x)=-\ln(1-|x|),\, x\in B^d(1).$ 
		
		Now,  consider the function $u: B^d(1) \to \mathbb{R}$ defined by $\displaystyle u(x)=\frac{|x|}{(1-|x|)^\frac{1}{t}}=e^{\frac{d_F(0,x)}{t}}\left(1-e^{-d_F(0,x)}\right),$ where $t$ is a parameter.  
		A direct calculation yields that 
		$$Du(x)=\frac{1}{t}e^\frac{d_F(0,x)}{t}\left[1+(t-1)e^{-d_F(0,x)}\right]Dd_F(0,x).$$
		By applying \eqref{tavolsag-derivalt}, we have that $F^*(x,Dd_F(0,x))=1$ for a.e. $x \in B^d(1)$, thus  
		\begin{align*}
			\|u\|_{W^{1,p}_F(B^d(1))}^p=&\left(\frac{1}{t}\right)^p\int_{B^d(1)} e^{\frac{p\cdot d_F(0,x)}{t}}\left[1+(t-1)e^{-d_F(0,x)}\right]^p \mathrm{d}V_F(x)\\&+\int_{B^d(1)} e^{\frac{p\cdot d_F(0,x)}{t}}\left(1-e^{-d_F(0,x)}\right)^p \mathrm{d}V_F(x).
		\end{align*}
		Therefore, since $\mathrm{d}V_F(x)=\mathrm{d}x$, we have that 
		\begin{align*}
			\|u\|_{W^{1,p}_F(B^d(1))}^p 
			&= \omega_{d-1}\left(\frac{1}{t}\right)^p\int_{0}^{1}\frac{(t-(t-1)s)^p}{(1-s)^{\frac{p}{t}}}s^{d-1}\mathrm{d}s \ + \ \omega_{d-1}\int_0^1\frac{s^p}{(1-s)^{\frac{p}{t}}}\cdot s^{d-1}\mathrm{d}s \\
			&\leq  \omega_{d-1} \int_{0}^{1} s^{d-1} (1-s)^{-\frac{p}{t}}\mathrm{d}s \ + \ \omega_{d-1}\int_0^1 s^{p+d-1} (1-s)^{-\frac{p}{t}} \mathrm{d}s\\
			&= \omega_{d-1} \left[ \mathtt{B}\left(d,1-\frac{p}{t}\right) + \mathtt{B}\left(p+d,1-\frac{p}{t}\right) \right].
		\end{align*}	
		where $\mathtt{B}$ denotes the Euler-Beta function.  

		\textbf{(\textbf{S}) \& (\textbf{MT}) cases:} Notice that  $$\displaystyle \|u\|_{L^q(B^d(1))}^q=\omega_{d-1}\mathtt{B}\left(q+d,1-\frac{q}{t}\right).$$
		Thus, if we choose $t\coloneqq \frac{p+q}{2}$, it turns out that $\|u\|_{W_F^{1,p}(B^d(1))}<\infty$ whereas $\|u\|_{L^q(B^d(1))}=+\infty$, i.e. $u\in W_F^{1,p}(B^d(1))\setminus L^q(B^d(1))$.
		
		\textbf{(\textbf{M}) case:} Let $t\coloneqq\frac{p^2}{d}>1$, and since $p>d$, it follows that
		$\|u\|_{W_F^{1,p}(B^d(1))}<\infty$. On the other hand, it is clear that $\|u\|_{L^\infty(B^d(1))}=+\infty$. In particular, it follows that $u\in W_F^{1,p}(B^d(1))\setminus L^\infty(B^d(1))$.
	\end{example}

	\subsection[Application: Elliptic PDE on Randers spaces]{Application: Multiple solutions for an elliptic PDE on Randers spaces } \label{sec:Application}

	In order to prove Theorem \ref{alkalmazas}, we recall an abstract tool, which is the following critical point result of Bonanno \cite{Bonanno} (which is
	actually a refinement of a general principle of Ricceri  \cite{Ricceri1,Ricceri2}):
	
	\begin{theorem}[\cite{Bonanno}, Theorem 2.1]\label{Bonanno}
		Let $X$ be a separable and reflexive real Banach space, and let $\Phi, J : X \to \mathbb{R}$ be two continuously G\^ateaux differentiable functionals, such that $\Phi(u) \geq 0$ for every $u\in X$. Assume that there exist $u_0, u_1 \in X$ and $\rho > 0$ such that 
		\begin{enumerate}
			\item[(1)]\label{nulladikF} $\Phi(u_0) = J(u_0) = 0$,
			\item[(2)]\label{elsoF} $\rho<\Phi(u_1)$,
			\item[(3)]\label{masodikF} $\displaystyle \sup_{\Phi(u)<\rho}J(u)<\rho \frac{J(u_1)}{\Phi(u_1)}$.
		\end{enumerate}
		Further, put $$\overline{a}=\zeta \rho \left(\rho \frac{J(u_1)}{\Phi(u_1)}-\sup_{\Phi(u)<\rho}J(u)\right)^{-1}, \ \mbox{ where } \zeta>1,$$
		and assume that the functional $\Phi - \lambda J$ is sequentially weakly lower semicontinuous, satisfies the Palais-Smale condition and
		\begin{enumerate}
			\item[(4)] $\displaystyle \lim_{\|u\|\to \infty}(\Phi(u) - \lambda J(u))=+\infty$, for all $\lambda\in [0, \overline{a}]$.
		\end{enumerate}
		Then there exists an open interval $\Lambda \subset [0, \overline{a}]$ and a number $\mu > 0$ such that for each $\lambda \in \Lambda$, the equation
		$\Phi'(u)-\lambda J'(u) = 0$ admits at least three solutions in $X$ having norm less than $\mu$.
	\end{theorem}

	For every $\lambda>0$ we define the energy functional associated with problem \eqref{feladat} as
	$${E}_\lambda: W^{1,p}_{F}(M)\to \mathbb{R}, \quad {E}_\lambda(u)=\Phi_0(u) - \lambda J_0(u),$$ where
	$$\Phi_0(u)=\frac{1}{p}\int_M{F^*}^p(x,Du(x))\,\mathrm{d}V_F(x) \ \ \mbox{ and } \ \ J_0(u)=\int_M \alpha(x)H(u(x))\,\mathrm{d}V_F(x).$$
	
	Since $(M,F)$ is a Randers space with $\displaystyle a \coloneqq \sup_{x\in M}\|\beta\|_g(x)<1$, the reversibility constant $r_F$ is finite, thus  $W^{1,p}_{F}(M)$ is a separable and reflexive Banach space, see Farkas, Krist\'aly, and Varga \cite{FarkasKristalyVarga_Calc}.

	Having in our mind Theorem \ref{Finsler}, we restrict the energy functional to the space $W^{1,p}_{F,G}(M)$. For simplicity, in the following we denote $$\mathcal{E}_\lambda = {E}_\lambda \rvert_{W^{1,p}_{F,G}(M)}, \ \Phi=\Phi_0\rvert_{W^{1,p}_{F,G}(M)},\ \mbox{ and } J=J_0\rvert_{W^{1,p}_{F,G}(M)}.$$
	
	In the sequel we prove that the energy functional ${E}_\lambda$ is $G$-invariant. Note that the $G$-invariance of the energy functional is an important tool in proving our theorem.

	\begin{lemma}\label{iso-lemma}
		Let $G$ be a compact connected subgroup of $\mathrm{Isom}_F(M)$ with $\mathrm{Fix}_M(G)=\{x_0\}$ for some $x_0\in M$. Then ${E}_\lambda$ is $G$-invariant, i.e., for every $\xi \in G$ and $u\in W^{1,p}_{F}(M)$ one has ${E_\lambda}(\xi u)={E}_\lambda(u)$. 	
	\end{lemma}
	
	\begin{proof}
		First we focus on the $G$-invariance of the functional $\Phi_0$. Since $\xi \in G$, we have that (see Deng and Hou \cite{DengPJM}) 
		\begin{equation}\label{isome-finsler}
			F(\xi x, d\xi_x(X))=F(x,X),\ \forall x\in M, X\in T_xM.
		\end{equation}
		Since $(\xi u)(x)=u(\xi^{-1}x)$, by  the chain rule, one has
		\begin{align}
			\Phi_0(\xi u)&=\int_M {F^*}^p(x,D(\xi u)(x))\,\mathrm{d}V_F(x)\nonumber\\&=\int_M {F^*}^p(x,D(u(\xi^{-1}x)))\,\mathrm{d}V_F(x)\nonumber\\ &=\int_M {F^*}^p(x,D(u(\xi^{-1}x))d\xi^{-1}_{x})\,\mathrm{d}V_F(x) \ \ \ \ \ \ \ \ \ \ \ (\mbox{ change of var. }\xi^{-1}x=y)\nonumber\\&=\int_M {F^*}^p(\xi y,D(u(y))d\xi^{-1}_{\xi y}\,\mathrm{d}V_F(\xi y).\label{elsoFF}
		\end{align}
		Since $\xi \in G$, $\mathrm{d}V_F(\xi y)=\mathrm{d}V_F(y)$. On the other hand, by the definition of the polar transform \eqref{polar-transform} and relation \eqref{isome-finsler}, we have 
		\begin{align}
			F^*(\xi y,D(u(y))d\xi^{-1}_{\xi y})&=\sup_{w\in T_{\xi y}M\setminus\{0\}}\frac{D(u(y))d\xi^{-1}_{\xi y}(w)}{F(\xi y,w)}\ \ \ \ \ \ \ \ \ \ \ (w:=d\xi_y(z),\ z\in T_yM)\nonumber\\
			&=\sup_{z\in T_y M\setminus\{0\}}\frac{Du(y)d\xi^{-1}_{\xi y}(d\xi_ y(z))}{F(\xi y,d\xi_y(z))}\nonumber =\sup_{z\in T_y M\setminus\{0\}}\frac{Du(y)(z)}{F(y,z)}\\
			&=F^*(y,Du(y)).\label{masodikFF}
		\end{align}
		Combining \eqref{elsoFF} and \eqref{masodikFF}, we get the desired $G$-invariance of the functional $\Phi_0$. 
		
		Since $\xi \in G$ and $\alpha\in L^1(M) \cap L^\infty(M)$ is a non-zero, non-negative function which depends on $d_F(x_0, \cdot)$ and $\mathrm{Fix}_M(G)=\{x_0\}$, it turns out that for every $u\in W^{1,p}_{F}(M)$, we have $J_0(\xi u)={J}_0(u),$
		which concludes the proof.
	\end{proof}

	The principle of symmetric criticality of Palais (see Krist\'aly, R\u adulescu and Varga \cite[Theorem
	1.50]{KVR-book}) and Lemma \ref{iso-lemma} imply that the critical points of
	$\mathcal{E}_\lambda = {E}_\lambda \rvert_{W^{1,p}_{F,G}(M)}$ are also critical points of the original functional ${E}_\lambda$. Therefore, it is enough to find critical points of $\mathcal{E}_\lambda$.
	
	\begin{lemma}\label{coercivity}
		For every $\lambda \geq 0$, the functional ${E}_\lambda$ is coercive and bounded below.
	\end{lemma}
	
	\begin{proof}
		Using a McKean-type inequality (see for instance Yin and He \cite[Theorem 0.6]{Yin-He}), we have that 
		$$\lambda_{1,g} \coloneqq \inf_{u \in W^{1,p}_g(M)}\frac{\displaystyle \int_M|\nabla_{g} u|^p{\rm d}v_{g}}{\displaystyle \int_M |u|^p\,{\rm d}v_{g}}\geq \left(\frac{(d-1)\kappa}{p}\right)^p,$$ 
		therefore, 
		$$\int_M |\nabla_g u|^p{\rm d}v_{g}\geq \frac{(d-1)^p\kappa^p}{p^p+(d-1)^p\kappa^p}\|u\|^p_{W^{1,p}_g(M)},\ \ \ u \in W^{1,p}_g(M).$$ 
		Using \eqref{randers-volume}, \eqref{masikbecsles}, and denoting $c(d,a,p,\kappa) \coloneqq \frac{(1-a^2)^{(d+1)/2}}{(1+a)^p} \cdot \frac{(d-1)^p\kappa^p}{p^p+(d-1)^p\kappa^p}$ we obtain that 
		\begin{equation} \label{D-norma_osszehasonlitas}
			\int_M {F^*}^p(x,Du(x))\,\mathrm{d}V_F(x) \geq c(d,a,p,\kappa)\|u\|^p_{W^{1,p}_g(M)},\ \ \ u \in W^{1,p}_g(M).
		\end{equation}
		From \hyperref[f2]{($A_2$)}, for every $u \in W^{1,p}_g(M)$ it follows that 
		$${E}_\lambda(u)\geq \frac{c(d,a,p,\kappa)}{p}\|u\|^p_{W^{1,p}_g(M)}-\lambda C\|\alpha\|_{L^1(M)}\left(c_\infty\|u\|_{W^{1,p}_g(M)}+c_\infty^w \|u\|^w_{W^{1,p}_g(M)}\right).$$ Since $w<p$, the claim clearly follows. 
	\end{proof}

	\begin{lemma}\label{PS}
		For every $\lambda \geq 0$, $\mathcal{E}_\lambda$ satisfies the Palais–Smale condition on $W^{1,p}_{F,G}(M)$.
	\end{lemma}
	
	\begin{proof}
		Let $\{u_k\}_k$ be a sequence in $W^{1,p}_{F,G}(M)$ such that
		$\{\mathcal{E}_{\lambda}(u_k)\}_k$ is bounded and $\|
		\mathcal{E}'_{\lambda}(u_k)\|_{*}\to 0.$ 
		Since  $\mathcal{E}_{\lambda}$ is coercive, the sequence $\{u_k\}_k$ is bounded in $W^{1,p}_{F,G}(M)$.
		Therefore, up to a subsequence, $u_k\weak u$ weakly in $W^{1,p}_{F,G}(M)$ for some $u\in W^{1,p}_{F,G}(M)$. Hence, due to Theorem \ref{Finsler} and Theorem \ref{mainThm:Riemann1}, it follows that $u_k\to u$ strongly in $L^\infty(M)$. In particular, we have that
		\begin{equation}\label{PS-1}
			\mathcal E_\lambda'(u)(u-u_k) \to 0 \quad {\rm and} \quad  \mathcal{E}_\lambda'(u_k)(u-u_k) \to 0 \quad {\rm as} \quad k\to \infty.
		\end{equation}

		On the one hand, it is easy to verify that 
		\begin{align*}
			&\int_{M}(Du(x)-Du_k(x))(\boldsymbol{\nabla}_F
			u(x) {F^*}^{p-2}(x,Du(x))-\boldsymbol{\nabla}_F
			u_k(x){F^*}^{p-2}(x,Du_k(x)))\,{\text d}V_{F}(x) \\
			&=  \mathcal
			E_\lambda'(u)(u-u_k)- \mathcal E_\lambda'(u_k)(u-u_k)
			+\lambda\int_{M}\alpha(x)[h(u_k)-h(u)](u_k(x) - u(x))\,{\text d}V_{F}(x).
		\end{align*}
		
		On the other hand, we have 
		\begin{align}\label{eziskell}
			&\left|\int_{M}\alpha(x)[h(u_k)-h(u)](u_k(x)-u(x))\,{\text
				d}V_{F}(x)\right| \leq \nonumber \\
			&\leq 2\|\alpha\|_{L^1(M)}\cdot\max\{|h(s)|: |s|\leq \|u\|_{L^\infty(M)}+1\}\|u_k-u\|_{L^\infty(M)}.
		\end{align}
		
		The mean value theorem implies that for all $x\in M$, 
		$$
		(Du(x)-Du_k(x))(\boldsymbol{\nabla}_F u(x) {F^*}^{p-2}(x,Du(x))-\boldsymbol{\nabla}_F
		u_k(x){F^*}^{p-2}(x,Du_k(x)))\geq l_F F^{*p}(x,Du(x)-Du_k(x)),
		$$
		where $l_{F}$ is the uniformity constant associated to $F$ (see \eqref{uniformity_const}). Since $(M,F)$ is a Randers space with $\displaystyle a \coloneqq \sup_{x\in M}\|\beta\|_g(x)<1$, it follows that $l_F > 0$, thus $u_k\to u$ in $W^{1,p}_{F,G}(M)$, which proves the claim.
	\end{proof}

	\begin{lemma}\label{LSC}
		For every $\lambda\geq 0$ the functional $\mathcal{E}_\lambda$ is sequentially weakly lower semicontinuous.
	\end{lemma}
	\begin{proof}
		As a norm-type function, $\Phi$ is sequentially weakly lower semicontinuous, therefore it  suffices to prove that $J$ is sequentially weakly continuous. To this end, consider a sequence $\{u_n\}_n$ in $W^{1,p}_{F,G}(M)$ which converges weakly to $u\in W^{1,p}_{F,G}(M)$, and suppose that $$J(u_n)\cancel{\to} J(u_n)\ \mbox{ as }n\to \infty.$$ Thus, there exist $\varepsilon>0$ and a subsequence of $\{u_n\}_n$, denoted again by $\{u_n\}_n$, such that $u_n\to u$ in $L^\infty(M)$ and  $$0<\varepsilon\leq |J(u_n)-J(u)|,\ \mbox{ for every }n\in \mathbb{N}.$$ Thus, by the mean value theorem (see also \eqref{eziskell}), there exists $\theta_n\in(0,1)$ such that
		\begin{align*}
			0<\varepsilon&\leq \left|\langle J'(u+\theta_n(u_n-u)),u_n-u\rangle_{W^{1,p}_F}\right|\\ &\leq \int_M \alpha(x)|h(u+\theta_n(u_n-u))|\cdot |u_n-u|\,\mathrm{d}V_F(x)\\&\leq  \|\alpha\|_{L^1(M)}\max\{|h(s)|:|s|\leq \|u\|_{L^\infty(M)}+1\}\cdot\|u_n-u\|_{L^\infty(M)}.
		\end{align*}
		Note that the last term tends to $0$, which provides a contradiction.
	\end{proof}	
	
	\begin{proof}[Proof of Theorem \ref{alkalmazas}]
		Let $s_0 > 0$ be given by condition \hyperref[f1]{($A_1$)}. We recall that $\displaystyle \sup_{R>0}\underset{d_{F}(x_{0},x)\leq R}{\mathrm{essinf}}\alpha(x) > 0,$ thus we choose an $R>0$ such that $\alpha_R:=\underset{d_{F}(x_{0},x)\leq R}{\mathrm{essinf}}\alpha(x) > 0.$
		For $r<\displaystyle R\frac{1-a}{1+a}$, we define the following function 
		$$u_{s_0,R,r}=\begin{cases}
			0, & x\in M\setminus B_F(x_0,R)\\
			\frac{s_0}{R-r}(R-d_F(x_0,x)),&  x\in B_F(x_0,R)\setminus B_F(x_0,r)\\
			s_0,&  x\in B_F(x_0,r)
		\end{cases}$$
		
		Recall that $r_{F}>0$ is the reversibility constant on $(M,F)$, see \eqref{reverzibilis}, i.e.  by the eikonal identity \eqref{tavolsag-derivalt} we have that $\frac{1}{r_F}\leq F^*(x,-Dd_F(x_0,x))\leq r_F.$ Therefore, 
		\begin{align*}
			\left(\frac{s_0}{R-r}\right)^p \frac{1}{r_F^p}\left(\mathrm{Vol}_F(B_F(x_0,R))-\mathrm{Vol}_F(B_F(x_0,r))\right)&\leq \int_M {F^*}^p(x,Du_{s_0,R,r}(x))\,\mathrm{d}V_F(x) 
			\\& \leq \left(\frac{s_0}{R-r}\right)^p {r_F^p} \mathrm{Vol}_F(B_F(x_0,R)).
		\end{align*}
		
		%
		%


		By $0\leq u_{s_0,R,r}\leq s_0$ and hypothesis \hyperref[f1]{($A_1$)} we have that \begin{align*}
			J(u_{s_0,R,r})&=\int_M \alpha(x) H(u_{s_0,R,r}(x))\, \mathrm{d}V_F(x)=\int_{B_F(x_0,R)}\alpha(x) H(u_{s_0,R,r}(x))\, \mathrm{d}V_F(x)\\ &\geq H(s_0)\alpha_R \mathrm{Vol}_F(B_F(x_0,r))>0.
		\end{align*}
		
		On the other hand, by \hyperref[f3]{($A_3$)}, we may fix $s_1\in (0,1]$ and $C_1>0$ such that 
		$$H(s)\leq C_1 |s|^q,\ |s|<s_1.$$ 
		By \hyperref[f2]{($A_2$)} we have that 
		$$|H(s)|\leq C(1+|s|^{w-1})|s|\leq C\frac{1+s_1^{w-1}}{s_1^{q-1}}|s|^q,\ |s|\geq s_1.$$ 
		Choosing $C_2=\max\left\{C_1, C\frac{1+s_1^{w-1}}{s_1^{q-1}}\right\}$, we get that 
		$$H(s)\leq C_2|s|^q,\ s\in \mathbb{R}.$$ 
		Therefore, 
		\begin{equation}\label{becslesJre}
			J(u)=\int_M \alpha(x)H(u(x))\,\mathrm{d}V_F(x) \leq C_2 \|\alpha\|_{L^1(M)} c_\infty^q \|u\|_{W^{1,p}_g(M)}^q.
		\end{equation}
		We claim that 
		\begin{equation}\label{fontoshatarertekJ}
			\limsup_{\rho\to 0}\frac{\displaystyle \sup\left\{J(u):\int_M {F^*}^p(x,Du(x))\,\mathrm{d}V_F(x)\leq p\rho\right\}}{\rho}\leq 0.
		\end{equation}
		To prove the previous claim, first observe that by \eqref{D-norma_osszehasonlitas} we have that 
		$$\frac{\displaystyle \sup\left\{J(u):\int_M {F^*}^p(x,Du(x))\,\mathrm{d}V_F(x)\leq p\rho\right\}}{\rho}\leq \frac{\displaystyle \sup\left\{J(u):c(d,a,p,\kappa)\|u\|^p_{W^{1,p}_g(M)}\leq p\rho\right\}}{\rho}.$$  
		On account of \eqref{becslesJre}, we have that 
		$$ \frac{\displaystyle \sup\left\{J(u):c(d,a,p,\kappa)\|u\|^p_{W^{1,p}_g(M)}\leq p\rho\right\}}{\rho}\leq \frac{C_2 \|\alpha\|_{L^1(M)}c_\infty^q \left(\frac{p\rho}{c(d,a,p,\kappa)}\right)^\frac{q}{p}}{\rho}\to 0 \quad \text{when } \rho \to 0,$$  
		since $q>p$, which proves the claim. 
		
		By \eqref{fontoshatarertekJ}, we may chose $\rho_0>0$ such that 
		$$\rho_0< c(d,a,p,\kappa)\|u_{s_0,R,r}\|^p_{W^{1,p}_g(M)}\leq \int_M {F^*}^p(x,Du_{s_0,R,r}(x))\,\mathrm{d}V_F(x),$$ and 
		$$\frac{\displaystyle \sup\left\{J(u):\int_M {F^*}^p(x,Du(x))\,\mathrm{d}V_F(x)\leq p\rho_0\right\}}{\rho_0}< \frac{J(u_{s_0,R,r})}{\Phi(u_{s_0,R,r})}.$$ 
		
		In Theorem \ref{Bonanno} we choose $u_1=u_{s_0,R,r}$ and $u_0 = 0$, and observe that the hypotheses \hyperref[elsoF]{(1)} and \hyperref[masodikF]{(2)} are satisfied.  We define 
		$$\overline{a}=\frac{1+\rho_0}{\frac{J(u_{s_0,R,r})}{\Phi(u_{s_0,R,r})}-\frac{\sup\left\{J(u):\Phi(u)\leq\rho_0\right\}}{\rho_0}}.$$ 
		Taking into account Lemmas \ref{coercivity}, \ref{PS} and \ref{LSC}, all the assumptions of Theorem \ref{Bonanno} are verified. Thus there exists an open interval $\Lambda\subset [0,\overline{a}]$ and a number $\mu > 0$ such that for each $\lambda\in \Lambda$, the equation $\mathcal{E}_\lambda'(u)=\Phi'(u)-\lambda J'(u)$ admits at least three solutions in $W^{1,p}_{F,G}(M)$ having $W^{1,p}_{F}(M)$-norms less than $\mu$. This concludes the proof.
	\end{proof}

	\section*{Acknowledgment}
	The authors are supported by the National Research, Development and Innovation Fund of Hungary, financed under the K\_18 funding scheme, Project No. 127926. C. Farkas is also supported by the Sapientia Foundation - Institute for Scientific Research,
	Romania, Project No. 17/11.06.2019.
	

\end{document}